\documentclass[fleqn,preprint,11pt]{elsarticle}

\makeatletter
\def\ps@pprintTitle{%
 \let\@oddhead\@empty
 \let\@evenhead\@empty
 \def\@oddfoot{\centerline{\thepage}}%
 \let\@evenfoot\@oddfoot}
\makeatother

\usepackage[margin=3cm]{geometry}    

\usepackage{latexsym,amsmath,amssymb,amsthm}
\usepackage{bm}
\usepackage{url}
\usepackage{color}
\usepackage{enumitem}

\newtheorem{theorem}{Theorem}
\newtheorem{lemma}{Lemma}
\newtheorem{corollary}{Corollary}
\newtheorem{pro}{Proposition}

\theoremstyle{definition}
\newtheorem{definition}{Definition}
\newtheorem{remark}{Remark}

\usepackage{tabularx}
\usepackage{algorithm}
\usepackage{algorithmic}

\newcommand{\bi}{\begin{itemize}}
\newcommand{\ei}{\end{itemize}}
\newcommand{\ben}{\begin{enumerate}}
\newcommand{\een}{\end{enumerate}}

\newcommand{\be}{\begin{equation}}
\newcommand{\ee}{\end{equation}}
\newcommand{\ba}{\begin{aligned}}
\newcommand{\ea}{\end{aligned}}
\newcommand{\bea}{\begin{eqnarray}}
\newcommand{\eea}{\end{eqnarray}}

\newcommand{\bc}{\begin{center}}
\newcommand{\ec}{\end{center}}

\newcommand{\ie}{{\it i.e.\ }}

\newcommand{\mbf}[1]{{\bm #1}}           

\newcommand{\sfrac}[2]{\mbox{\small $\frac{#1}{#2}$}}
\newcommand{\half}{\sfrac{1}{2}}
\newcommand{\rhalf}{\frac{1}{2}}
\newcommand{\bigO}{{\mathcal O}}

\newcommand{\RR}{\mathbb{R}}

\newcommand{\bsigma}{\bm{\sigma}}

\newcommand\dt{\Delta t}
\newcommand\dx{\Delta x}

\newcommand\bnu{{\boldsymbol{\nu}}}
\newcommand\bxi{{\boldsymbol{\xi}}}
\newcommand\bx{\mathbf{x}}
\newcommand\by{\mathbf{y}}
\newcommand\bz{\mathbf{z}}

\newcommand{\bsigmahat}{\hat{\bm{\sigma}}(\bxi)}
\newcommand{\bfhat}{\hat{\bm{f}}(\bxi)}
\newcommand{\Vhat}{\hat{V}(\bxi)}

\newcommand\bxy{\bx-\by}

\newcommand\lbxy{|\bxy|}
\newcommand\xdy{\bx\cdot\by}

\newcommand\dk{\frac{\partial G(\bx-\by,t-\tau)}{\partial \bnu(\by)}}
\newcommand\snd{\frac{\partial G(\bx-\by,t-\tau)}{\partial \bnu(\bx)}}
\newcommand{\dkj}[1]{\frac{\partial G(\bx-\by,#1-\tau)}{\partial \bnu(\by)}}
\newcommand\ldk{\frac{\partial G_{\rm L}(\bx-\by)}{\partial \bnu(\by)}}

\newcommand\cs{\mathcal{S}}
\newcommand\cd{\mathcal{D}}
\newcommand\cvol{\mathcal{V}}
\newcommand\cinit{\mathcal{I}}
\newcommand\cV{\mathcal{V}}

\newcommand\sd{{S^{d-1}}}
\newcommand\sj[1]{{S^{#1}}}

\newcommand\ttau{t-\tau}

\newcommand\sigmanm{\sigma^{nm}}
\newcommand\ynm{Y^{m}_{n}}

\def\FI{Flatiron Institute, Simons Foundation, New York, New York 10010}

\def\njit{Department of Mathematics Sciences,
New Jersey Institute of Technology,
Newark, New Jersey 07102}

\def\nyu{Courant Institute of Mathematical Sciences,
  New York University, New York, New York 10012}

\def\upenn{Department of Mathematics,
  University of Pennsylvania, 209 South 33rd Street,
  Philadelphia, PA 19104}

\def\papertitle{Explicit unconditionally stable methods 
for the heat equation via potential theory}

\begin{document}

\begin{frontmatter}

\title{\papertitle}


\author{Alex Barnett\fnref{fi}}
\address[fi]{\FI}
\ead{abarnett@flatironinstitute.org}

\author{Charles L. Epstein\fnref{upenn}}
\address[upenn]{\upenn}
\ead{cle@math.upenn.edu}

\author{Leslie Greengard\fnref{fi,nyu}}
\address[nyu]{\nyu}
\ead{greengard@courant.nyu.edu}

\author{Shidong Jiang\fnref{njit}}
\address[njit]{\njit}
\ead{shidong.jiang@njit.edu}

\author{Jun Wang\fnref{fi}}
\ead{jwang@flatironinstitute.org}
 

\begin{abstract}
  We study the stability properties of explicit
  marching schemes for second-kind Volterra integral equations that
  arise when solving boundary value problems for the heat equation by
  means of potential theory.  It is well known that explicit finite
  difference or finite element schemes for the heat equation are stable
  only if the time step $\dt$ is of the order $\mathcal{O}(\dx^2)$,
  where $\dx$ is the finest spatial grid spacing. In contrast, for the Dirichlet and Neumann problems on the unit ball in all dimensions $d\ge 1$, we show
  that the simplest Volterra marching scheme, i.e., the forward Euler scheme,
  is {\em unconditionally stable}. Our proof is based on an explicit
  spectral radius bound of the marching matrix, leading to an estimate
  that an $L^2$-norm of the solution to the integral equation
  is bounded by $c_dT^{d/2}$ times the norm of the right hand side.
  For the Robin problem on the half space in any dimension,
  with constant Robin (heat transfer) coefficient $\kappa$,
  we exhibit a constant $C$ such that the forward Euler scheme
  is stable if $\dt < C/\kappa^2$, independent of any spatial discretization.
  This relies on new lower bounds
  on the spectrum of real symmetric Toeplitz
  matrices defined by convex sequences.
  Finally, we show that the forward Euler
  scheme is unconditionally stable for the Dirichlet problem on any smooth {\em convex} domain in any dimension, in $L^\infty$-norm.
\end{abstract}

\begin{keyword}
  heat equation \sep
  Abel equation \sep 
  forward Euler scheme \sep 
  Volterra integral equation \sep 
  stability analysis \sep
  Toeplitz matrix \sep
  convex sequence\sep
  modified Bessel function of the first kind
\end{keyword}

\end{frontmatter}
\section{Introduction}

In this paper, we study the stability of integral equation methods for the heat equation
\begin{align}
\frac{\partial u}{\partial t}(\bx, t) - \alpha\Delta u(\bx,t) &= F(\bx,t) 
\label{heateq0} \\
u(\bx,0) &= u_0(\bx) \nonumber
\end{align}
for $0 \leq t \leq T$, subject to suitable boundary conditions, in a
smooth domain $D\subset \mathbb{R}^d$.  Without loss of generality, we will
assume that the diffusion coefficient (thermal conductivity) $\alpha$ is one
in most of our discussion. We consider three standard boundary conditions:
the Dirichlet boundary condition
\be
u(\bx,t) = f(\bx,t) |_{\bx\in \Gamma,\ t>0}, \qquad \bx \in \Gamma=\partial D,
\label{dirichletbc}
\ee
the Neumann boundary condition
\be
\frac{\partial u(\bx,t)}{\partial \bnu_\bx} = g(\bx,t) |_{\bx\in \Gamma,\ t>0}, \qquad \bx \in \Gamma,
\label{neumannbc}
\ee
and 
the Robin boundary condition
\be
\label{robinbc}
\frac{\partial u(\bx,t)}{\partial \bnu_\bx} + \kappa u(\bx,t) = h(\bx,t) |_{\bx\in \Gamma,\ t>0},
\qquad \bx \in \Gamma,
\ee
Here, $\kappa>0$ is the {\em heat transfer coefficient}, and \eqref{robinbc}
models heat transfer via Newton's law of cooling \cite{crank}. For all three
boundary conditions, we assume that proper compatibility conditions are
satisfied between the initial and boundary data.

Before turning to the integral equation framework, we briefly review the finite
difference approach. For this, we assume we are given a spatial mesh discretizing 
the domain $D$ with grid points $x_n$ and seek to approximate the solution 
$u_m^n \approx u(x_n,t_m)$ at time steps $t_0,t_1,\dots,t_N$
with $t_m=m\dt$.
Two of the simplest schemes for solving \eqref{heateq0} are 
the forward and backward Euler methods:
\[
 \frac{u_n^{m+1} - u_n^m}{\dt} = \Delta_h[u]^m_n + F(x_n,t_m)
\]
and
\[
 \frac{u_n^{m+1} - u_n^m}{\dt} = \Delta_h[u]^{m+1}_n + F(x_n,t_m)~,
\]
respectively. Here 
$\Delta_h[u]^m_n$ denotes the finite difference 
approximation of the Laplacian evaluated at the grid point $x_n$ at time $t_m$.
It is well known that 
the backward Euler scheme is unconditionally stable, while, in $d$ dimensions, the forward Euler scheme requires that the time step satisfy the  condition
$\dt < \frac{1}{2d} \dx^2$, for the case of the standard 2nd-order finite difference Laplacian stencil on a uniform spatial grid with step size $\dx$ in each direction
(see, for example, \cite[p.~158]{thomas}). The constraint changes when using less standard stencils.
For nonuniform grids, the time step restriction is more complicated to analyze, but 
generally requires that $\dt = \mathcal{O}(h_{min}^2)$ where $h_{min}$ is the finest mesh spacing
in the discretization. 

The backward Euler scheme is {\em implicit} and requires the solution of a large
sparse linear system at each time step $t_m$.
The forward Euler scheme, on the other hand, is {\em explicit} and inexpensive. The stability
restriction, however, 
forces extremely small time steps to be taken, making long-time simulations
impractical. This has spurred the development
of a variety of alternative approaches, including locally one-dimensional schemes,
alternating direction implicit methods, etc. \cite{adi}.

When finite difference methods are used to solve general 
initial-boundary value problems, GKSO (Gustafsson--Kreiss--Sundstr{\" o}m--Osher)
theory plays a critical role 
~\cite{gustafsson,gustafsson1972mcom,osher1969mcom,strikwerda,trefethen1983jcp},
and requires that
the interior marching scheme be Cauchy stable (that is, beyond the stability condition above,
the discrete boundary conditions must satisfy additional criteria).
In short, stability imposes rather intricate constraints on the coupling between the
interior marching scheme and the boundary conditions themselves.
Similar considerations are involved when using finite element methods.

An alternative to direct discretization of the governing PDE is to recast the
problem as a boundary integral equation using heat potentials \cite{kress2014,pogorzelski}.
The Green's function for the heat equation is
\be
G(\bx,t)
=\frac{1}{(4\pi t)^{d/2}}e^{-\frac{|\bx|^2}{4t}},
\quad \bx \in \mathbb{R}^d~.
\ee
We assume that the boundary $\Gamma$ of $D$ is at
least $C^2$, and let $\sigma$ be a square integrable function on $\Gamma \times [0,T]$.
Then the single layer heat potential $\cs$ is defined by the formula
\be\label{slpdef}
\cs[\sigma](\bx,t) = \int_0^t\int_\Gamma
G(\bx-\by,t-\tau)\sigma(\by,\tau)ds(\by)d\tau
\ee
and the double layer heat potential $\cd$ is defined by 
\be\label{dlpdef}
\cd[\sigma](\bx,t) = \int_0^t\int_\Gamma
\dk
\sigma(\by,\tau)ds(\by)d\tau,
\ee
where $\bnu(\by)$ is the unit outward normal vector at $\by\in \Gamma$.
The initial potential is defined by 
\be\label{ipotdef}
\cinit[u_0](\bx,t) = \int_D
G(\bx-\by,t-\tau)u_0(\by)d\by
\ee
and the volume potential is defined by 
\be\label{vpotdef}
\cvol[F](\bx,t) = \int_0^t \int_D
G(\bx-\by,t-\tau) F(\by,\tau) d\by d\tau.
\ee
By the linearity of the problem, we may decompose the solution
into
\be
u(\bx,t)=u^{(F)}(\bx,t)+u^{(B)}(\bx,t),
\ee
where all initial and volume data is captured by the free-space volume forced term
\be
u^{(F)}(\bx,t)= \cinit[u_0](\bx,t) + \cvol[F](\bx,t)~,
\label{uF}
\ee
while $u^{(B)}$ is the solution to a pure boundary value problem
with zero initial data, zero volume forcing, and modified boundary data.
Note that $u^{(F)}$ needs only {\em evaluation}
of initial and volume potentials; it requires no linear solve.
Thus, there is no stability
issue with $u^{(F)}$, and its error is simply the quadrature error
in evaluating the integrals that appear in \eqref{ipotdef} and \eqref{vpotdef}.
In other words, unlike finite difference or finite element methods,
the volume part is completely decoupled from the boundary part
in integral equation methods from the perspective of stability analysis.

For the Dirichlet problem, we proceed by representing $u^{(B)}(\bx,t)$ as a double layer
potential with unknown density $\sigma$. The jump relation (see section~\ref{sec:potentials}) then leads to
the following second kind Volterra equation,
\be\label{bie}
\left(-\half+\cd \right)[\sigma](\bx,t)= 
\tilde{f}(\bx,t), \qquad (\bx,t)\in \Gamma\times [0,T],
\ee
where $\cd$ is interpreted in a principal value sense,
and the corrected data is
\[ \tilde{f}(\bx,t) \;:=\; f(\bx,t) - u^{(F)}(\bx,t)~,
\qquad \bx\in\Gamma~.
\]

The main objective of this paper is to demonstrate certain advantages of integral
equation methods by giving, for several combinations of archetypal geometries and boundary conditions, rigorous stability bounds for
the simplest explicit time marching scheme, namely the forward Euler scheme.
This scheme is derived by assuming $\sigma(\by,t)$
is piecewise constant over each time interval $[j\dt,(j+1)\dt)$,
taking on the value $\sigma(\by,j\dt)$.
For \eqref{bie}, this leads to a marching scheme of the form 
\begin{align}
 \sigma(\bx, n \dt) \;=\; & 2 \sum_{j = 0}^{n-1}
 \int_{j\dt}^{(j+1)\dt} \int_\Gamma
\frac{\partial G(\bx-\by,n\dt -\tau)}{\partial \bnu(\by)}
\sigma(\by,j\dt)ds(\by)d\tau \nonumber \\
&\qquad - \; 2\tilde{f}(\bx, n\dt)~,  \qquad \qquad n=1,2,\dots
\label{biemarch}
\end{align}
This falls into the class of collocation
schemes \cite[Sec.~13.3]{kress2014},
as well as convolution quadrature schemes \cite{lubich1986ima}.
It is explicit, since $\sigma(\bx,n\dt)$ does not
appear on the right-hand side.
It is also first-order accurate (e.g.\ see section \eqref{s:robin1d}).
For the Neumann and Robin problems, second kind Volterra
equations are obtained by representing $u^{(B)}(\bx,t)$
instead as a single layer potential; other than a change of kernel, the forward Euler scheme remains the same.

The principal reasons that integral equation
methods have received relatively little attention for solving the heat equation 
has been that direct evaluation of layer (or volume) potentials require quadratic
work in the total number of unknowns as well as the design of suitable quadrature
rules. Recent advances in fast algorithms
for heat potentials, however, have removed this obstacle.
We refer the reader to 
the papers 
\cite{greengard1990cpam,greengard1991sisc,greengard1998fgtnew,powerheat,
lubichheat2,lubichheat1,strain_adapheat,tauschheat1,wang2017nyu,wang2018,fullheatsolver} 
and the references therein for further discussion. 

We now summarize the results in this paper.
Perhaps the simplest geometry is the half-space
$D=\mathbb{R}^d_+ := \{\bx=(x_1,x_2,\cdots,x_d)\in \mathbb{R}^d\, |\, x_d\ge 0\}$
with $\Gamma=\partial D=\mathbb{R}^{d-1}$.
It is easy to check that the integral kernel of $\mathcal{D}$ 
is identically zero on $\Gamma$, so \eqref{bie} reduces to
\be
\sigma(\bx,t)=-2 \tilde f (\bx,t)~.
\label{trivial}
\ee
This is an analytic solution, so that stability follows trivially.
A similar trivial analytic solution arises when the single layer potential is
used to solve the Neumann problem on the half-space.
Thus, we consider the Dirichlet and Neumann problems on possibly
the next-simplest domain,
the unit ball $B^d\subset \mathbb{R}^d$ (i.e., $\Gamma$ is
the unit sphere $S^{d-1}$).
For both these latter cases, we show that the forward Euler scheme
is unconditionally stable in all dimensions $d\ge 1$. Specifically, we show that
for $T\ge 1$, 
\be
\|\sigma\|_2 \le c_d T^{d/2} \|\tilde{f}\|_2
\label{diribound}
\ee
for all $N$, $\dt$ such that $N\dt \le T$.
Here $N$ is the total number of time steps, $\dt$ is the time step size,
$\|\cdot \|_2$ denotes a space-time $L^2$-norm,%
\footnote{Explicitly, $\|\sigma\|_2^2 := \sum_{j=0}^N \int_\Gamma \sigma(\bx,j\dt)^2 ds(\by)$, i.e.\ the norm is $l^2$ in time $[0,T]$ but $L^2$ over the surface $\Gamma$.}
and $c_d$ is a positive constant depending on $d$.
The estimate \eqref{diribound} is obtained by a
Gershgorin spectral radius bound of the marching matrix; we show
that this is no longer tight for the Dirichlet problem if a fairly
mild condition is imposed on $\dt$. Indeed, we are able to show the
improved estimate in two dimensions,
\be
\|\sigma\|_2 \le 7 \|\tilde{f}\|_2
\label{diribound2}
\ee
for $\dt\le 1$ and any $N$.

Returning to the $d$-dimensional half-space, the simplest  boundary condition for which the integral equation is non-trivial is the Robin condition. 
We show that here the forward Euler scheme
has a time step restriction determined by the
physical parameter $\kappa$, namely $\dt < \frac{\pi}{c^2 \kappa^2}$
with $c =3-\sqrt{2}$.
Finally, considering more general domains, we prove that the forward Euler scheme for the Dirichlet problem is unconditionally stable  for smooth convex domains in all dimensions, in the $L^\infty$-norm.

Firstly, in section \ref{sec:potentials} we summarize the necessary
properties of layer potentials.  Then in section \ref{sec:normbounds}
we present a lower bound for the spectrum of a Toeplitz operator
defined by a convex sequence; this will be needed later to handle
cases where the sequences are not summable and thus Gershgorin is
inapplicable.  The Dirichlet and Neumann problems on the unit ball are
then treated in section \ref{sec:diri}, the Robin problem on the half
space in section \ref{sec:robin}, and  the Dirichlet problem on
$C^1$ convex domains in section \ref{sec:convex}. We conclude in
section~\ref{s:conc}. Finally, an appendix covers estimates on special
functions used in the body of the paper.

\section{Properties of heat potentials} \label{sec:potentials}

By construction, the single and double layer heat potentials
\eqref{slpdef} and
\eqref{dlpdef} satisfy the heat equation. They also satisfy certain
well-known jump conditions when the target point $\bx$ approaches the boundary
from either side \cite{kress2014,pogorzelski}.
In particular, for $\bx_0 \in \Gamma$,
the normal derivative of the single layer potential
$\cs[\sigma]$ satisfies the relation
\be\label{snjump}
\lim_{\epsilon\rightarrow 0+}
\frac{\partial \cs[\sigma](\bx_0 \pm \epsilon \bnu(\bx_0 ),t)}{\partial \bnu(\bx_0)}
= \mp\half\sigma(\bx_0,t) + \cs_{\bnu}[\sigma](\bx_0,t),
\ee
and the double layer potential 
$\cd[\sigma]$ satisfies the relation
\be\label{dlpjump}
\lim_{\epsilon\rightarrow 0+}
\cd[\sigma](\bx_0 \pm \epsilon \bnu(\bx_0 ),t)
= \pm\half\sigma(\bx_0,t) + \cd[\sigma](\bx_0,t),
\ee
where both $\cs_{\bnu}[\sigma](\bx_0,t)$ and $\cd[\sigma](\bx_0,t)$
are interpreted in the Cauchy principal value sense.
If we represent the solution to the heat equation
\eqref{heateq0} via a double layer potential
$u(\bx,t)=\mathcal{D}[\sigma](\bx,t)$, then the 
integral equation \eqref{bie} follows immediately from the jump relation \eqref{dlpjump}.

The kernel of the double layer potential is given explicitly by
\be\label{dkernel}
\dk
=\frac{(\bxy)\cdot\bnu(\by)}{2^{d+1}\pi^{d/2}(t-\tau)^{1+d/2}}e^{-\frac{\lbxy^2}{4(t-\tau)}}
\ee
and the kernel of $\cs_{\bnu}$ is given by 
\[
\snd 
=-\frac{(\bxy)\cdot\bnu(\bx)}{2^{d+1}\pi^{d/2}(t-\tau)^{1+d/2}}e^{-\frac{\lbxy^2}{4(t-\tau)}}.
\]

Finally, the initial potential \eqref{ipotdef} is well known to satisfy 
the homogeneous heat equation
with initial data $u_0(\bx)$, while the volume potential 
\eqref{vpotdef} satisfies the inhomogeneous heat equation 
\[
\frac{\partial u}{\partial t}(\bx, t) - \Delta u(\bx,t) = F(\bx,t)
\]
with zero initial data.

\begin{remark}
Using these properties, it is straightforward to see that representing the solution
to the Dirichlet problem in the form
\[ u(\bx,t) = \cd[\sigma](\bx,t) + \cinit[u_0](\bx,t) + \cvol[F](\bx,t) \]
leads to the integral
equation \eqref{bie}, with the only unknown corresponding to the double layer density
$\sigma$.
\end{remark}

\begin{remark}
On the unit sphere $\sd$, $\bnu(\by)=\by$ and $|\bx|=|\by|=1$.
Thus, $(\bxy)\cdot\bnu(\by)=-(1-\xdy)$, $\lbxy^2=2(1-\xdy)$,
and \eqref{dkernel} reduces to
\be\label{dksphere}
\dk
=-\frac{1-\xdy}{2^{d+1}\pi^{d/2}(t-\tau)^{1+d/2}}e^{-\frac{1-\xdy}{2(t-\tau)}}.
\ee
\end{remark}

\section{Spectral bounds for real symmetric Toeplitz operators} \label{sec:normbounds}

Although for many of the later results we can use simple
Gershgorin spectral bounds for matrices, for the tight bound for the zeroth mode of the $d=2$ Dirichlet disc (section~\ref{sec:tighter}), and the Robin case in the half-space (section~\ref{sec:robin}),
a more delicate spectral bound on Toeplitz matrices is needed.

Let $S^1$ be the unit circle in the complex plane, parametrized by polar angle $\theta$
with normalized arc length measure $d\lambda=\frac{1}{2\pi}d\theta$.
For any $f$ in the Hilbert space $L^2(S^1)$, we write
\be
f(\theta)=\sum_{n=-\infty}^\infty f_n e^{in\theta},
\ee
in terms of the orthogonal basis $\{ e^{in\theta} \}_{n\in\mathbb{Z}}$,
where $f_n$ ($n\in\mathbb{Z}$) is the $n$th Fourier coefficient of
$f$ defined by
\[
f_n=\frac{1}{2\pi}\int_0^{2\pi}f(\theta)e^{-in\theta}d\theta.
\]
The Hardy space $H^2$ is defined by
\[
H^2=\{f\in L^2(S^1) \, | \, f_n=0, n<0\},
\]
and we let $P$ denote the orthogonal projection of $L^2(S^1)$ onto $H^2$.
The Toeplitz operator $T_f: H^2\rightarrow H^2$ with symbol $f\in L^\infty(S^1)$,
is defined by 
\[
T_f(u)=P(fu)~.
\]

The operator $T_f$ is closely related to an 
infinite-dimensional Toeplitz matrix with entries $t_{ij},\ i,j \in \mathbb{N}$ that satisfy
$t_{ij}=t_{i+1,j+1}$ for all $i,j$. That is, the matrix is constant along diagonals and
determined by a two-sided sequence $(t_n)_{n\in\mathbb{Z}}$ with $t_{ij}=t_{i-j}$.
The Fourier transform maps $T_f$ onto the class of Toeplitz matrices on $l^2(\mathbb{Z}_+)$;
that is, if $\left(T_f(u)\right)_n$ denotes the $n$th Fourier coefficient of 
$T_f(u)$, then
\[
\bigl(T_f(u)\bigr)_n= \left\{\begin{array}{ll}
\sum_{m=0}^\infty f_{n-m}u_m~,& n \ge 0 \\
0~, & n< 0 \end{array}\right.
\]
where $u_m$ is the $m$th Fourier coefficient of $u$.
\begin{definition}
  A sequence $\{a_n\}_{n\in\mathbb{Z_+}}$ is said to be convex if
  $\delta^2a_n\geq 0$ for every $n>0$, where $\delta^2a_n := a_{n-1}-2a_n+a_{n+1}$ is the central second difference.
  \end{definition}
Recall that for $n\in\mathbb{Z_+}$ the Fej\'{e}r kernel $F_n(x)$ is defined to be
  \[
  F_n(\theta)=\sum_{j=-n}^{n}\left(1-\frac{|j|}{n+1}\right)e^{ij\theta}=\frac{1}{n+1}\left[\frac{\sin\left(\frac{n+1}{2}\theta\right)}{\sin\left(\frac{\theta}{2}\right)}\right]^2.
  \]

The following theorem can be found in \cite[Chapter 1, Theorem 4.1]{katznelson}.
\begin{theorem}\label{convexfourier}
  If $a_n\rightarrow 0$ and the sequence $\{a_n\}_{n\in\mathbb{Z}_+}$
  is convex, then the series
  \be
  v(\theta)=\sum_{n=1}^{\infty}n\, (\delta^2 a_n) \, F_{n-1}(\theta)
  \label{toepl_symb}
  \ee
  converges in $L^1([-\pi,\pi])$ to a non-negative function, which is continuous except at $0,$ such that
  $v_n=a_n$.
\end{theorem}
It is often the case that the function $v(\theta)$ blows up as $\theta\to 0.$ Using the elementary estimate on the Fej\'{e}r kernel
\[
F_n(\theta)\leq \min\left\{(n+1),\frac{\pi^2}{(n+1)\theta^2}\right\},
\]
\cite[Chapter 1, formula 3.10]{katznelson} and the fact that, for a convex sequence tending 
to zero, we have $\lim_{n\to\infty}n(a_n-a_{n+1})=0,$
one can show that
\be\label{eqn10.00}
\lim_{\theta\to 0}\theta v(\theta)=0~.
\ee

Bounds on the spectrum of finite Toeplitz matrices are of interest in
many applications \cite{dembo1988,hertz1992,laudadio2008,melman1999}.
When a real symmetric Toeplitz operator (or matrix) is generated by a
positive sequence, the Gershgorin circle theorem~\cite[\S3.3]{thomas}
often gives a satisfactory upper bound on its spectral radius or the
largest eigenvalue. Curiously, satisfactory lower bounds on the
smallest eigenvalue do not seem to be available.  The following
theorem leads to a tight lower bound on the smallest eigenvalue of a
real symmetric Toeplitz matrix, defined by a convex sequence, even
when the operator it defines is unbounded.
\begin{theorem}\label{lowerbound}
Suppose that
  $\{v_n\}_{n\in\mathbb{N}}$ is a convex sequence and $\lim_{n\rightarrow\infty}v_n=0.$ Set $v_0=2v_1-v_2,$ and let $v(\theta)$ be the non-negative function defined by the sequence
  $\{v_n\}_{n\in\mathbb{Z}_+}$ as in Theorem~\ref{convexfourier}.
  Suppose that $V$ is the self-adjoint Toeplitz matrix defined by 
  $V_{ii}=0$ and $V_{ij}=v_{|i-j|}$. Then,
  for any $\mbf{u}\in\mathbb{C}^N$, we have the lower bound
\[
  \langle V\mbf{u},\mbf{u}\rangle\geq (v_2-2v_1)\|\mbf{u}\|^2.
\]
\end{theorem}
\begin{proof}        
  For a finite length vector $\mbf{u}=(u_0,\dots,u_N,0,0,\dots)_{n\in\mathbb{Z}_+},$ define the function
  \be
  u(\theta)=\sum_{n=0}^N u_n e^{in\theta}.
  \ee
  Theorem~\ref{convexfourier} implies  that
\[
  \begin{split}
 0\leq \frac{1}{2\pi}\int_0^{2\pi}v(\theta)|u(\theta)|^2d\theta &=\frac{1}{2\pi}\int_0^{2\pi}v(\theta)
 \sum_{0\leq j,k\leq N}u_j\bar{u}_k e^{i(j-k)\theta}d\theta\\
 &=\sum_{0\leq j,k\leq N}v_{k-j}u_j\bar{u}_k\\
  &=\langle V\mbf{u},\mbf{u}\rangle+ (2v_1-v_2)\|\mbf{u}\|^2.
   \end{split}
\]
\end{proof}   
\begin{remark}\label{finitetoeplitz}
  If $V_N$ is the upper left $N\times N$ principal submatrix of $V$, then,
  by an application of the Rayleigh--Ritz theorem, its spectrum is bounded below by $(v_2-2v_1).$
\end{remark}
\begin{remark}
  For certain applications the sequence, $\{v_n\}_{n\in\mathbb{N}}$,
  generating $T_v$ is not convex.  In this case, one may consider an
  operator of the form, $cI+T_v+T_a$ with $c$ and
  $\{a_n\}_{n\in\mathbb{N}}$ chosen so that $(c, v_1+a_1, v_2+a_2,
  \ldots)$ is a convex sequence. If $T_a$ is a bounded operator, then
  the previous theorem implies a lower bound on the spectrum of $V$
  \[
 \langle V\mbf{u},\mbf{u}\rangle  \geq -(c +\|T_a\|)\|\mbf{u}\|^2\text{ for }\mbf{u}\in\mathbb{C}^N.
  \]
\end{remark}
\begin{remark}
  If the function $v(\theta)$ defined in~\eqref{toepl_symb} is unbounded, then the Toeplitz operator, $T_v,$ it defines is not a
  bounded operator, and is not defined on all of $H^2.$ The discussion above easily applies to show that this operator is defined on a dense subset, and its closure is self-adjoint:
  Equation~\eqref{eqn10.00} implies that if $u\in H^2,$ then
  $v(1-e^{i\theta})u\in L^2.$ Thus, $T_vw=P(vw)\in H^2,$ for $w$ in the
  subspace $(1-e^{i\theta})H^2.$ It is not difficult to see that this
  subspace is dense. If $u\in H^2$ and $r>1,$ then
\[
\left(\frac{1-e^{i\theta}}{r-e^{i\theta}}\right)u\in H^2
\]
and
\[
\lim_{r\to 1^+}\left\|\left(\frac{1-e^{i\theta}}{r-e^{i\theta}}\right)u-u\right\|_2=0.
\]
Since $\langle T_v w,w\rangle\geq 0,$ for $w$ in this domain, the
Friedrichs extension of $T_v$ is a closed self-adjoint, non-negative
operator defined on a dense subspace $D_v\subset H^2.$
\end{remark}

\section{The Dirichlet and Neumann problems on the unit ball $B^d$} \label{sec:diri}

We consider first the forward Euler scheme \eqref{biemarch}
for the Dirichlet problem \eqref{bie}.
For general $d\ge 1$,
our approximation of the unknown density $\sigma$ is piecewise
constant in time,
\[
\sigma(\by,\tau)=\sigma(\by,t_j)=\sigma_j(\by),
\quad \tau\in [t_j,t_{j+1}) \qquad {\rm for}\ j=0,1,\ldots.
\]
where $t_j=j\dt$. We restate \eqref{biemarch} in the form
\be\label{heatrec}
-\half\sigma_j(\bx)+\sum_{k=0}^{j-1}V_{j-k}[\sigma_k](\bx)=f_j(\bx)
:=f(\bx,j\dt),
\ee
for $j=0,1,2,\ldots$,
where the tilde has been dropped from $f$, and where
the action of each spatial integral operator $V_{j-k}: C(\Gamma)\to C(\Gamma)$
is defined by
\[
V_{j-k}[\sigma_k](\bx)= \int_\Gamma \cV_{j-k}(\bx,\by)
\sigma_k(\by) ds(\by) ~.
\]
Here the operator kernel is itself the integral of the heat kernel over
one time-step,
\[
\cV_{j-k}(\bx,\by)=\int_{k\dt}^{(k+1)\dt}\dkj{j\dt}d\tau.
\]
Due to time-shift invariance, a simpler way to write the spatial kernel is
\[
\cV_l(\bx,\by)=\int_{0}^{\dt}\dkj{l\dt}d\tau, \quad l\geq 1,~
\]
and $\cV_0(\bx,\by)$ is set identically to $0$. For initialization
of time-stepping we set $\sigma_0 \equiv f_0 \equiv 0$.

\subsection{The Dirichlet problem in one dimension} \label{sec:1ddirichlet}

The boundary $\Gamma$ of the unit ball in one dimension consists of
only two points $\bx=\pm 1$.
Let the time-stepped density at these two points be $\bsigma^\pm=\{\sigma_j^\pm\}_{j=0}^N$,
and the data $\mbf{f}^\pm=\{f^\pm_j\}_{j=0}^N$.
We will stack each pair into a single column, e.g.\ $[\bsigma^-\, ,\, \bsigma^+]^T$.
Recalling \eqref{trivial}, the density at each boundary point is
trivially coupled to the data at that same point; however,
the coupling to the other boundary point will involve the double layer kernel
acting at a distance of 2.
Thus, \eqref{heatrec} becomes a $2\times 2$ system with trivial diagonal
blocks and Toeplitz off-diagonal blocks. Namely, after $N$ time-steps
the stacked vectors are related by,
\be\label{mode1d}
\left[\begin{matrix} -\frac{I}{2}& V\\ V& -\frac{I}{2}\end{matrix}\right]
\left[\begin{matrix} \bsigma^- \\ \bsigma^+\end{matrix}\right]=
\left[\begin{matrix} \mbf{f}^-\\ \mbf{f}^+\end{matrix}\right]~,
\ee
where $I$ is the size-$(N+1)$ identity matrix.
Here the action of the lower-triangular Toeplitz matrix $V\in\mathbb{R}^{(N+1)\times(N+1)}$
is given by
\[
[V\bsigma^{\pm}]_j=\sum_{k=0}^{j-1}v_{j-k}\sigma_k^{\pm}
\qquad \text{ for }j=1,\dots,N,
\]
with the convolution coefficients $\{v_l\}$  given by 
\be\label{vcoef1d}
v_l=-
\int_{0}^{\dt}
\gamma(l\dt-\tau)d\tau, \quad l\geq 1~, \quad \text{ and } v_0=0~.
\ee
Here the underlying kernel is the double layer acting at a distance of 2,
\be
\gamma(t) \; :=\;
\frac{1}{2\sqrt{\pi}} t^{-3/2}e^{-\frac{1}{t}}, \qquad   t > 0~.
\label{gamma1d}
\ee
We denote the symmetric part of $V$ 
by $W$, and make its dependence on $N$ and $\dt$ explicit, thus
\be\label{w1d}
W(N;\dt) \; :=\;  \frac{V+V^T}{2}.
\ee

We have the following lemma.
\begin{lemma}
  Fix $T>0$. Then, for any $N$ and $\dt$ with $N\dt\leq T$,
  the spectral radius $\rho(N;\dt)$ of the matrix $W(N;\dt)$
  has the bound
  \be
  \rho(N;\dt) \;\le\; C_1(T) ~,
  \label{maxbound1d}
  \ee
where
\be
C_1(T)\;:= \;
\int_{0}^{T} \gamma(T-\tau)d\tau
\; = \; \frac{1}{2\sqrt{\pi}}
\int_{\frac{1}{T}}^{\infty}\frac{1}{\sqrt{u}}e^{-u}du \;<\; \rhalf ~.
\label{C1T}
\ee
\label{l:C1T}
\end{lemma}
\begin{proof}
  Using the Gershgorin circle theorem~\cite[\S3.3]{thomas}, and the
  fact that the diagonal entries of $W$ are all zero, we have
\be
\rho(N;\dt) \;\le\;
\max_i\sum_{j=1}^{N+1} |w_{ij}|\leq 2\sum_{l=1}^{N} \half |v_l|
\le \sum_{l=1}^{N} |v_l|
~.
\label{rhobound1d}
\ee
Now setting $t=N\dt$, we may collapse this sum into a single integral
\[
\ba
\sum_{l=1}^{N} |v_l |&= \sum_{l=1}^{N} \int_0^{\dt}\gamma(l\dt-\tau)d\tau
=\sum_{k=1}^N\int_{0}^{\dt}\gamma(N\dt-(k-1)\dt-\tau)d\tau\\
&=\sum_{k=1}^N\int_{(k-1)\dt}^{k\dt}\gamma(N\dt-\tau)d\tau
=\int_{0}^{N\dt}\gamma(N\dt-\tau)d\tau = C_1(N \dt)
\ea
\]
according to the definition \eqref{C1T} of the function $C_1$.
Combining the last two results we have $\rho(N;\dt) \le C_1(N\dt)$.
The expression in \eqref{C1T} follows from
the change
of variables $u=\frac{1}{T-\tau}$. A further change of variables
$x=\sqrt{u}$ leads to
\be
C_1(T) = \frac{1}{\sqrt{\pi}}\int_{\frac{1}{\sqrt{T}}}^{\infty}e^{-x^2}dx
 < \frac{1}{\sqrt{\pi}}\int_{0}^{\infty}e^{-x^2}dx
=\rhalf ~, \qquad \mbox{ for all } T>0~.
\ee
Finally, the above expression shows that $C_1(T)$ is a monotonically
non-decreasing function of $T$, so that
$\rho(N;\dt) \le C_1(N\dt) \le C_1(T)$.
\end{proof}
It is clear from \eqref{mode1d} that
to get a stability bound we need to control
the gap between $C_1(T)$ and $\half$. For $T\ge 1$,
this turns out to shrink only polynomially in $T$:
\be
\rhalf-C_1(T)=\frac{1}{2\sqrt{\pi}}
\int_0^{\frac{1}{T}}\frac{1}{\sqrt{u}}e^{-u}du
> \frac{1}{2e\sqrt{\pi}}
\int_0^{\frac{1}{T}}\frac{1}{\sqrt{u}}du
= \frac{1}{e\sqrt{\pi T}}~.
\label{c1tbound}
\ee
This very easily allows us to prove the following. 
\begin{theorem} \label{1dthm}
Suppose that $T\geq 1$. Then, using $\|.\|$ for the $l^2$-norm in $\mathbb{R}^{2(N+1)}$,
\be\label{1ddiribound}
\|[\bsigma^-,\bsigma^+]\|
\;\leq\;
e\sqrt{\pi T}\|[\mbf{f}^+,\mbf{f}^-]\|
\ee
for all $N$, $\dt$ such that $N\dt\leq T$.
That is, for the $d=1$ unit ball where $\Gamma = \{-1,1\}$,
the forward Euler scheme
  for solving the second kind Volterra integral equation \eqref{bie}
  is unconditionally stable on any finite time interval $[0,T]$.
\end{theorem}
\begin{proof}
We use a technique that will recur throughout this paper: we take the inner product of \eqref{mode1d}
with $-[\bsigma^-,\bsigma^+]^T$, giving
\be
\label{heatnorm1d}
\frac{1}{2}\|[\bsigma^+,\bsigma^-]\|^2-2\langle W\bsigma^+,\bsigma^-\rangle
\;=\;
-(\langle\mbf{f}^+,\bsigma^+\rangle+\langle\mbf{f}^-,\bsigma^-\rangle)
~.
\ee
Applying the Cauchy--Schwarz inequality  and \eqref{maxbound1d} to the second term on the left side, and Cauchy--Schwarz to the right hand side, we obtain
\[
\frac{1}{2}\|[\bsigma^+,\bsigma^-]\|^2-2C_1(T)\|\bsigma^+\|\cdot\|\bsigma^-\|\leq \|[\bsigma^+,\bsigma^-]\|\cdot\|[\mbf{f}^+,\mbf{f}^-]\|
\]
Using the arithmetic-geometric mean inequality on the left hand side of this estimate gives
\[\left(\frac{1}{2}-C_1(T)\right)\|[\bsigma^+,\bsigma^-]\|^2\leq \|[\bsigma^+,\bsigma^-]\|\cdot\|[\mbf{f}^+,\mbf{f}^-]\|.
\]
Finally dividing by $\left(\frac{1}{2}-C_1(T)\right)\|[\bsigma^+,\bsigma^-]\|$ and applying \eqref{c1tbound} gives
\[
\|[\bsigma^-,\bsigma^+]\|\leq \frac{1}{\half -C_1(T)}\|[\mbf{f}^+,\mbf{f}^-]\|
\leq e\sqrt{\pi T}\|[\mbf{f}^+,\mbf{f}^-]\|~,
\]
which completes the proof.
\end{proof}

\subsection{The Dirichlet problem in two dimensions} \label{sec:2d}

We now consider \eqref{heatrec} when $\Gamma$ is the unit circle $\sj{1}$.
We decompose both $\sigma_j(\by)$ and $f_j(\bx)$ into Fourier series:
\[
\ba
\sigma_j(\by)&=\sum_{n=-\infty}^{+\infty}\sigma_j^ne^{in\phi},
\quad \by=(\cos\phi,\sin\phi),\\
f_j(\bx)&=\sum_{n=-\infty}^{+\infty}f_j^ne^{in\theta},
\quad \bx=(\cos\theta,\sin\theta).
\ea
\]
From \eqref{dksphere}, writing $s=\theta-\phi$, the $n$th Fourier mode of the
kernel is
\begin{eqnarray}
\int_{\sj{1}}\dk
e^{in\phi} d\phi
&=&\int_0^{2\pi}
-\frac{1-\cos(\theta-\phi)}{8\pi(t-\tau)^{2}}e^{-\frac{1-\cos(\theta-\phi)}{2(t-\tau)}}
e^{in\phi} d\phi
\nonumber \\
&=&-\gamma_n(t-\tau)e^{in\theta},
\label{2dfh}
\end{eqnarray}
where, noting that the imaginary part of $e^{-ins}$ cancels by symmetry,
we have
\be
\gamma_n(t) \; :=\;
\frac{1}{8\pi t^2}\int_0^{2\pi}
(1-\cos(s))
e^{-\frac{1-\cos(s)}{2t}}
\cos(ns) ds~,  \qquad   t > 0~.
\label{gamman2d}
\ee
Since $\{ e^{in\theta} \}$ are orthonormal, each Fourier mode evolves independently.
The marching scheme (or recurrence) \eqref{heatrec} for the $n$th mode is then
\be\label{moden2d}
-\half\sigma_j^n-\sum_{k=0}^{j-1}v^n_{j-k}\sigma_k^n=f_j^n, \qquad j=0,1,2,\ldots ,
\ee
where the convolution coefficient $v^n_l$ is given by the formula
\be\label{vncoef2d}
v^n_l=\int_{0}^{\dt}\gamma_n(l\dt-\tau)d\tau, \quad l\geq 1,
\ee
and we set $v^n_0=0$.
The system \eqref{moden2d} for $j=0,1,\cdots,N$ can be written in matrix-vector form
\be\label{moden2dm}
\left(-\half I-V^n\right)\bsigma^n=\mbf{f}^n,
\ee
where $I$ is the $(N+1)\times (N+1)$ identity matrix, 
$V^n\in \mathbb{R}^{(N+1)\times(N+1)}$
with entries $v^n_{j,k}=v^n_{j-k}$, $\bsigma^n=\{\sigma_j^n\}_{j=0}^N$,
$\mbf{f}^n=\{f^n_j\}_{j=0}^N$.
As before, the symmetric part of $V^n$ we denote by
$W^n$,
\be\label{wn2d}
W^n(N;\dt) \; :=\;  \frac{V^n+(V^n)^T}{2} \;= \;
\frac{1}{2}
\begin{bmatrix}0 & v_1^n & v_2^n & \ldots & v_{N}^n \\
v_1^n & 0 & v_1^n & \ldots & v_{N-1}^n &\\
\vdots&\vdots&\vdots&\vdots&\vdots\\
v_{N}^n&v_{N-1}^n&\ldots& v_1^n&0
\end{bmatrix}.
\ee

\subsubsection{Stability analysis}

We now prove two key results. The first is that the forward Euler scheme
is unconditionally stable for any fixed time interval $[0,T]$
(Theorem \ref{2dthm1}).
The second is that, when $\dt < 1$,
we have the stronger result that
the $L^2$-norm of the solution is bounded by a constant multiple of the $L^2$-norm
of $f$.
We require the following lemma.

\begin{lemma}
  Fix $T>0$. Then, for any $N$ and $\dt$ with $N\dt\leq T$,
  and all $n\in \mathbb{Z}$,
  the spectral radius $\rho_n(N;\dt)$ of the matrix $W^n(N;\dt)$
  has the bound
  \be
  \rho_n(N;\dt) \;\le\; C_2(T) ~,
  \label{maxbound}
  \ee
where, in terms of the definition \eqref{gamman2d},
\be
C_2(T)\;:= \;
\int_{0}^{T} \gamma_0(T-\tau)d\tau
\; = \; \frac{1}{4\pi}\int_0^{2\pi}e^{-\frac{1-\cos(s)}{2T}}ds \;<\; \half ~.
\label{C2T}
\ee
\label{l:C2T}
\end{lemma}
\begin{proof}
  Let $n\in \mathbb{Z}$.
  Since the integrand in \eqref{gamman2d}, excluding the $\cos ns$ factor,
  is non-negative,
  we observe that $|\gamma_n(t)|\leq \gamma_0(t)$, so
  $|v^n_l|\leq v^0_l$ for all $l\ge 1$.
  Using this, the Gershgorin theorem, and the
  fact that the diagonal entries of $W^n$ are all zero, we have
\be
\rho_n(N;\dt) \;\le\;
\max_i\sum_{j=1}^{N+1} |w^n_{ij}|\leq 2\sum_{l=1}^{N} \half |v^n_l|
\le \sum_{l=1}^{N} v^0_l
~.
\label{rhonbound1}
\ee
Now setting $t=N\dt$, we may collapse this sum into a single integral
\[
\ba
\sum_{l=1}^{N} v^0_l &= \sum_{l=1}^{N} \int_0^{\dt}\gamma_0(l\dt-\tau)d\tau
=\sum_{k=1}^N\int_{0}^{\dt}\gamma_0(N\dt-(k-1)\dt-\tau)d\tau\\
&=\sum_{k=1}^N\int_{(k-1)\dt}^{k\dt}\gamma_0(N\dt-\tau)d\tau
=\int_{0}^{N\dt}\gamma_0(N\dt-\tau)d\tau = C_2(N \dt)
\ea
\]
according to the definition \eqref{C2T} of the function $C_2$.
Combining the last two results we have $\rho_n(N;\dt) \le C_2(N\dt)$.
To prove the expression in \eqref{C2T} we
insert \eqref{gamman2d},
interchange the order of integration and apply the change
of variables $\lambda=\frac{1-\cos(s)}{2(T-\tau)}$, thus
\[
\ba
C_2(T) &:= \int_{0}^{T}\gamma_0(T-\tau)d\tau
=\int_{0}^{T}
\frac{1}{8\pi(T-\tau)^{2}}\int_0^{2\pi}
(1-\cos(s))
e^{-\frac{1-\cos(s)}{2(T-\tau)}}ds d\tau
\\
& = \frac{1}{4\pi}\int_0^{2\pi}e^{-\frac{1-\cos(s)}{2T}}ds
<\half ~, \qquad \mbox{ for all } T>0~.
\ea
\]
Finally, the above expression shows that $C_2(T)$ is a monotonically
increasing function of $T$, so that
$\rho_n(N;\dt) \le C_2(N\dt) \le C_2(T)$.
\end{proof}

Analogous to before, it is clear from \eqref{moden2dm} that
to get a stability bound we need to bound from below
the gap between $C_2(T)$ and $\half$.
This motivates the following.

\begin{pro}
  \be\label{c2t}
  C_2(T)\;=\;\half e^{-\frac{1}{2T}}I_0\left(\frac{1}{2T}\right)~,
  \ee
  where $I_n(\cdot)$
  is the modified regular Bessel function of order $n$ (see Appendix).
  For $T\geq 1$,
  \be\label{c2tbound}
  \half -C_2(T) \; \geq\;  \frac{1}{10T}~.
  \ee
  \label{p:10T}
\end{pro}
\begin{proof}
  \eqref{c2t} follows from the integral representation
 of $I_0(x)$ \eqref{inrep}.
 \eqref{c2tbound} follows from the facts that $I_0(x)\leq 1+\frac{x^2}{2}$ \cite[\S10.25.2]{nisthandbook} and $e^{-x}\leq 1-\frac{x}{2}$ for $x\leq 1$.
\end{proof}

\begin{theorem} \label{2dthm1}
Suppose that $T\geq 1$. Then for all $n\in \mathbb{Z}$,
\be\label{2dbound}
\|\bsigma^n\| \; \leq \; 10T\|\mbf{f}^n\|
\ee
for all $N$, $\dt$ such that $N\dt\leq T$.
That is, when $\Gamma$ is the unit circle $\sj{1}$, the forward Euler scheme
  for solving the second kind Volterra integral equation \eqref{bie}
  is unconditionally stable on any finite time interval $[0,T]$.
\end{theorem}
\begin{proof}
Since we are working in the Fourier domain, $\bsigma^n$ and $\mbf{f}^n$ are
complex-valued. Thus,
we split \eqref{moden2dm} into two independent real systems
\be\label{modek3}
\ba
\left(-\half I-V^n\right)\bsigma_r^n&=\mbf{f}_r^n,\\
\left(-\half I-V^n\right)\bsigma_i^n&=\mbf{f}_i^n,
\ea
\ee
where $\bsigma_r^n$ and $\bsigma_i^n$ are the real and imaginary part
of $\bsigma^n$, respectively.

Multiplying both sides of the first equation in \eqref{modek3}
by $-(\bsigma_r^n)^T$, we have
\be
\half \|\bsigma_r^n\|^2+(\bsigma_r^n)^TV^n\bsigma_r^n=
\half \|\bsigma_r^n\|^2+(\bsigma_r^n)^TW^n\bsigma_r^n=-(\bsigma_r^n)^T\mbf{f}_r^n~.
\label{heatnorm}
\ee
Applying \eqref{maxbound} on the left side \eqref{heatnorm} and
the Cauchy--Schwartz inequality on the right side, we obtain
\[
\left(\half -C_2(T)\right)\|\bsigma_r^n\|^2\leq \half \|\bsigma_r^n\|^2+(\bsigma_r^n)^TW^n\bsigma_r^n=-(\bsigma_r^n)^T\mbf{f}_r^n \leq \|\bsigma_r^n\| \|\mbf{f}_r^n\|~.
\]
That is, finally applying Proposition~\ref{p:10T},
\[
\|\bsigma_r^n\|\leq \frac{1}{\half -C_2(T)}\|\mbf{f}_r^n\|
\leq 10T\|\mbf{f}_r^n\|~.
\]
Similar result holds for $\|\bsigma_i^n\|$. As
$\|\bsigma^n\|=\sqrt{\|\bsigma_r^n\|^2+\|\bsigma_i^n\|^2},$ these two
inequalities give \eqref{2dbound}.
\end{proof}
%
%
\subsubsection{Tighter bounds}
\label{sec:tighter}

We now show that the dependence on $T$ in \eqref{2dbound} can be
removed when the time step satisfies $\dt\leq 1$. This is a 
physically reasonable requirement since we have assumed that the 
diffusion coefficient is one, and the domain has area of order one.
We first provide a bound on $\rho_n(N;\dt)$ for $n\neq 0$
that is independent of the total time $N\dt$.

\begin{lemma}
  Let $N$ and $\dt>0$ be arbitrary,
  and let $\rho_n(N;\dt)$ be the spectral radius of $W^n(N;\dt)$ defined in
  \eqref{wn2d}. Then for all $n\neq 0$,
  \be\label{rhonbound2}
  \rho_n(N;\dt) \; \leq\; \frac{1}{2|n|+1} ~.
  \ee
\end{lemma}
\begin{proof}
Clearly, it is sufficient to prove \eqref{rhonbound2} for $n>0$.
For this, let us note that
substituting \eqref{gamman2d} into \eqref{vncoef2d}, exchanging
the order of integration, and making the change of variables
$\lambda=\frac{1-\cos(s)}{2(l\dt-\tau)}$, we obtain
\[
v^n_l
=\left\{\begin{array}{ll}
\frac{1}{4\pi}\int_0^{2\pi}e^{-\frac{1-\cos(s)}{2\dt}}\cos(ns)ds, & l=1,\\
\frac{1}{4\pi}\int_0^{2\pi}\left(e^{-\frac{1-\cos(s)}{2l\dt}}
-e^{-\frac{1-\cos(s)}{2(l-1)\dt}}\right)\cos(ns)ds,& l>1.
\end{array}
\right.
\]
By the integral representation~\eqref{inrep} of $I_n$, we have
\be\label{vncoef2}
v^n_l
=\left\{\begin{array}{ll}
\half e^{-\frac{1}{2\dt}}I_{|n|}(\frac{1}{2\dt}), & l=1,\\
\half \left(
e^{-\frac{1}{2l\dt}}I_{|n|}(\frac{1}{2l\dt})-e^{-\frac{1}{2(l-1)\dt}}I_{|n|}(\frac{1}{2(l-1)\dt})
\right),& l>1.
\end{array}
\right.
\ee
From \eqref{vncoef2}, defining $i_n(x):=e^{-x}I_n(x)$ and $x_l=1/(2l\dt)$,
we consider the sum
  \be\label{rhonbound3}
  S_n= 2\sum_{l=1}^{N} |v^n_l|=i_n(x_1)+\sum_{l=2}^N
  |i_n(x_l)-i_n(x_{l-1})| ~.
  \ee
  By Lemma \ref{inbound1} (see Appendix), the function $i_n(x)$ assumes its 
  unique maximum at $r_n>0$; it increases monotonically
  on $[0,r_n]$ and decreases monotonically on $[r_n,+\infty).$ We now consider
  \eqref{rhonbound3} on a case-by-case basis.
\begin{enumerate}[label=(\alph*)]
\item All $x_l$ lie on $[0,r_n]$:
  Since $x_l<x_{l-1}$ and $i_n(x)$ increases on $[0,r_n]$, we have
\[
  \ba
  S_n&\leq i_n(x_1)-\sum_{l=2}^N
  (i_n(x_l)-i_n(x_{l-1}))=2i_n(x_1)-i_n(x_N)\\
  &\leq 2i_n(x_1)<\frac{2}{2n+1}.
  \ea
\]
  where the last inequality follows from \eqref{inbound2}.
\item All $x_l$ lie on $[r_n,\infty)$: In this case, we have
\[
  S_n\leq i_n(x_1)+\sum_{l=2}^N
  (i_n(x_l)-i_n(x_{l-1}))=i_n(x_N)<\frac{1}{2n+1}.
\]
\item $x_1>\cdots>x_m\geq r_n>x_{m+1}>\cdots>x_N$: In this case, we have
\[
  \ba
  S_n&\leq i_n(x_1)+\sum_{l=2}^m(i_n(x_l)-i_n(x_{l-1}))\\
  &\quad +|i_n(x_m)-i_n(x_{m+1})|
  -\sum_{l=m+2}^N(i_n(x_l)-i_n(x_{l-1}))\\
  &=i_n(x_m)+|i_n(x_m)-i_n(x_{m+1})|+i_n(x_{m+1})-i_n(x_N)\\
  &\quad <i_n(x_m)+|i_n(x_m)-i_n(x_{m+1})|+i_n(x_{m+1})\\
  &=2\max(i_n(x_m),i_n(x_{m+1}))\\
  &\quad <\frac{2}{2n+1}.
  \ea
\]
  \end{enumerate}

  By \eqref{rhonbound1} we have
  \[
  \rho_n(N;\dt) \leq \sum_{l=1}^{N} |v^n_l|=\half S_n<\frac{1}{2n+1},
  \]
  completing the proof.
\end{proof}
\begin{corollary}
  For all $n\neq 0$,
  \[
  \|\bsigma^n\|\leq \frac{1}{\half -\frac{1}{2|n|+1}}\|\mbf{f}^n\|
  \leq 6\|\mbf{f}^n\|.
  \]
\end{corollary}

Thus, all non-zero modes are unconditionally stable.
The zeroth Fourier mode is a bit more subtle,
and requires the convex sequence results of section~\ref{sec:normbounds}.
It brings in a weak restriction on $\dt$, as follows.
\begin{lemma}\label{rho0bound2}
  Suppose that $a=0.05$ and $\dt\leq 1$. Then $c_2I+W^0+aW^1$ is a positive
  definite matrix  if
  \be
  c_2= \half e^{-1/2}I_0\left(\half \right)
  +\sfrac{1}{6}a 
  \;\approx\; 0.33085\ldots
  \ee
\end{lemma}
%
\begin{proof}
  Define the sequence $y_j=\half (v_j^0+a v_j^1)$ for
  $j\geq 1$ and $y_0=c_2$.
  Theorem~\ref{lowerbound} then
  shows that a sufficient condition for the positive
  semi-definiteness of $c_2I+W^0+aW^1$ is that the sequence $\{y_j\}_{j\in\mathbb{Z}_+}$ is convex.
  But $y_1=\frac{1}{4}f(x_1)$, $y_j=\frac{1}{4}(f(x_j)-f(x_{j-1}))$ ($j>1$),
  where $f$ is the function defined in Lemma~\ref{rho0bound1}, and
  $x_j=2j\dt$.
  That is, $y_j$ is the first order difference of $f$. Furthermore, the convexity
  of $\{y_j\}_{j\in\mathbb{N}}$ is equivalent to the non-negativity of the third order
  difference of $f$, which follows from the fact that $f'''(x)>0$ for all $x>0$
  as proved in Lemma~\ref{rho0bound1}.
  For $j=0$, the convexity of the sequence
  requires that one choose $c_2$ such that
  \be\label{2dconvex}
  c_2+y_2=y_0+y_2\geq 2y_1.
  \ee
  By the integral representation \eqref{inrep}
  of $I_0$, it is easy to see that $e^{-x}I_0(x)$ is strictly
decreasing.
Thus, we have
  $e^{-1/2}I_0\left(\half \right)\geq e^{-\frac{1}{2\dt}}I_0\left(\frac{1}{2\dt}\right)$
  for $\dt\leq 1$.
  Furthermore,
\[ \max_{[0,\infty)}e^{-x}I_1(x)<\sfrac{1}{3} \]
by \eqref{inbound2}. 
  Hence, \eqref{2dconvex} is achieved by choosing
  \[
  c_2=\half e^{-1/2}I_0\left(\half \right)+\sfrac{1}{6}a
  > 2y_1=\half e^{-\frac{1}{2\dt}}\left(I_0\left(\frac{1}{2\dt}\right)
  +I_1\left(\frac{1}{2\dt}\right)\right)
  \]
  for $\dt\leq 1$.   
\end{proof}
\begin{corollary}
  Suppose that $\dt\leq 1$. Then, for arbitrary $N$,
  \[
  \|\bsigma^0\| \; \leq\;  7\|\mbf{f}^0\|.
  \]
\end{corollary}
\begin{proof}
  Set $a=0.05$.
  By Lemma \ref{rho0bound2}, the smallest eigenvalue of $W^0$ is bounded
  by
  \[
  \lambda^0_{\rm min}\geq -c_2-a\lambda^1_{\rm max}
  \geq -c_2-a\rho_1\geq -c_2-\frac{1}{3}a.
  \]
  Thus a simple bound using the value of $c_2$ from Lemma \ref{rho0bound2}
  is
  \begin{align*}
    7 \|\bsigma^0\|^2 & \le
    \left(\half -c_2-\frac{1}{3}a\right)\|\bsigma^0\|^2 \leq
  \half \|\bsigma^0\|^2+(\bsigma^0)^TW^0\bsigma^0=-(\bsigma^0)^T\mbf{f}^0 
   \nonumber \\
  &\leq \|\bsigma^0\| \|\mbf{f}^0\|,
  \end{align*}
  completing the proof.
\end{proof}
\subsection{The Dirichlet problem in higher dimensions} \label{sec:hd}

In dimensions $d>2$, we consider the Dirichlet problem on the unit ball, with data specified on the unit sphere $\sd$.
The unknown density $\sigma$ is decomposed using the corresponding spherical  harmonics \cite{morimoto} 
\[
\sigma(\by,\tau)=\sum_{n=0}^\infty\sum_{m=1}^{a_{n,d}}\sigma^{nm}(\tau)\ynm(\by)~, \qquad \by \in S^{d-1} \subset \mathbb{R}^d,
\; \tau\ge 0
\]
where
\[
a_{n,d}=(2n+d-2)\frac{(n+d-3)!}{n!(d-2)!}.
\]
Here, $a_{n,d}$ is the dimension of $H_n(\sd)$, the space of
homogeneous harmonic polynomials of degree $n$ on $\mathbb{R}^d,$
whose restrictions to the unit sphere are spanned by $\{\ynm\},$ the
spherical harmonics of degree $n$. When $d=3$, $a_{n,d}=2n+1$, the
inner summation is usually written as $\sum_{m=-n}^n$, and the
spherical harmonics $\ynm(\theta,\phi)$ are defined by
\[
Y_n^m(\theta,\phi)=\sqrt{\frac{2n+1}{4\pi}}\sqrt{\frac{(n-|m|)!}{(n+|m|)!}}
P_n^{|m|}(\cos\theta)e^{im\phi},
\]
where $P_n^m(\cos \theta)$ is the associated Legendre polynomial~\cite[\S18.3]{nisthandbook} 
of degree $n$ and order $m$.

The spherical harmonics admit the following
integral representation \cite{morimoto}
\be\label{ynmrep}
\ynm(\bx)=\frac{a_{n,d}}{\omega_d}\int_{\sd} P_{n,d-1}(\bx\cdot\by)\ynm(\by)dS(\by),
\ee
where $\omega_d$ is the area of $\sd$ defined in \eqref{spherearea},
and the $P_{n,d-1}$ are Gegenbauer polynomials~\cite[Chapter 2]{morimoto} (also called
ultraspherical polynomials), defined by the Rodrigues formula
\be\label{gegenbauer}
P_{n,d-1}(t)=\frac{(-1)^n}{2^n}\frac{\Gamma\left(\frac{d-1}{2}\right)}
{\Gamma\left(n+\frac{d-1}{2}\right)}
\frac{1}{(1-t^2)^{\frac{d-3}{2}}}\frac{d^n}{dt^n}(1-t^2)^{n+\frac{d-3}{2}}.
\ee

The Funk--Hecke formula~\cite[Chapter 2, Theorem 2.39]{morimoto} states that
\be\label{fhformula}
\int_{\sd}f(\bx\cdot\bz)P_{n,d-1}(\by\cdot\bz)dS(\bz)= \beta_{n,d-1}P_{n,d-1}(\bx\cdot\by),
\ee
where
\[
\beta_{n,d-1}=\omega_{d-1}\int_{-1}^1P_{n,d-1}(t)f(t)(1-t^2)^{\frac{d-3}{2}}dt
\]
and $f$ is any measurable function such that
\[
\int_{-1}^1|f(t)|(1-t^2)^{\frac{d-3}{2}}dt \; < \; \infty ~.
\]
In $d=3$ this reduces to $f\in L^1[-1,1]$.

We compute the double layer heat potential $nm$th Fourier mode,
\be
\ba
&\int_{\sd}\frac{\partial G(\bx-\by,t-\tau)}{\partial \bnu(\by)}\ynm(\by)dS(\by)\\
&=-\int_{\sd}
\frac{1-\xdy}{2^{d+1}\pi^{d/2}(t-\tau)^{1+d/2}}e^{-\frac{1-\xdy}{2(t-\tau)}}
\ynm(\by)dS(\by)\\
&=-\frac{a_{n,d}}{\omega_d}\int_{\sd}\gamma_{n,d}(\ttau) P_{n,d-1}(\bx\cdot\bz)\ynm(\bz)dS(\bz)\\
&=-\gamma_{n,d}(\ttau)\ynm(\bx)~,
\label{ynmint}
\ea
\ee
where, by analogy with \eqref{gamman2d},
\be\label{gamman}
\gamma_{n,d}(t):=\frac{\omega_{d-1}}{2^{d+1}\pi^{d/2}t^{(d+2)/2}}
\int_{-1}^1 
(1-x)e^{-\frac{1-x}{2t}} P_{n,d-1}(x)(1-x^2)^{(d-3)/2}dx ~.
\ee
The third equality makes use of
\eqref{ynmrep}, \eqref{fhformula}, and exchanging the order of integration.
The last step follows again from \eqref{ynmrep}.
Notice that $\gamma_{n,d}$ does not depend on the order $m$.

Since the $\{ \ynm \}$ form an orthonormal basis for functions in $L^2(\sd)$ and
\eqref{ynmint} shows that each spherical harmonic evolves independently under
the action of the double layer heat potential operator,
we may consider the time evolution for each mode $nm$ separately.

For the forward Euler scheme, we again assume that $\sigma(\bx,t)$ takes the constant value
$\sigma_j(\bx)=\sigma(\bx,j\dt)$ over each interval $[j\dt,(j+1)\dt]$,
$j=0,1,\dots$.
Equivalently, each spherical harmonic mode $\sigmanm(t)$ takes the constant
value $\sigmanm_j=\sigmanm(j\dt)$ over the interval. A straightforward calculation leads to the following
recurrence for the $nm$th spherical harmonic mode,
analogous to \eqref{moden2d}:
\be\label{ndmodek}
-\half \mu_j-\sum_{k=0}^{j-1}v^n_{j-k}\mu_k=g_j, \qquad j=0,1,2,\ldots ,
\ee
where we use the abbreviations $\mu_j:=\sigmanm_j$, $g_j=f^{nm}_j$,
and the matrix elements
\be
v^n_l=\int_0^{\dt}\gamma_{n,d}(l\dt-\tau)d\tau, \qquad l>0,
\label{vnl}
\ee
involve the kernel modes \eqref{gamman}, and, as before, $v^n_0=0$.

\subsubsection{Stability analysis}
The normalization in \eqref{gegenbauer} leads to
\cite{morimoto1980,muller}
\[
|P_{n,d-1}(x)|\leq 1=P_{0,d-1}(x), \quad x\in [-1,1].
\]
As the other terms in \eqref{gamman} are non-negative, we have
\[
|\gamma_{n,d}(t-\tau)|\leq \gamma_{0,d}(t-\tau) ~, \qquad t-\tau > 0~.
\]
An almost identical proof as in Lemma~\ref{l:C2T} leads to the following lemma.
\begin{lemma}
  Fix $T>0$.
  Then, for any $N$ and $\dt$ with $N\dt \le T$, and all $n\in \mathbb{Z}_+$,
  the spectral radius $\rho_{n,d}(N;\dt)$, of the symmetric
  Toeplitz matrix $W^n(N;\dt)$
  as defined by \eqref{wn2d} with $v_l^n$ given by \eqref{vnl},
  has the bound
\[
\rho_{n,d}(N;\dt) \;\leq\; C_d(T)~,
\]
where
\[
C_d(T):=\int_0^T\frac{\omega_{d-1}}{2^{d+1}\pi^{d/2}(T-\tau)^{(d+2)/2}}
\int_{-1}^1 
\frac{(1-x)}{e^{\frac{1-x}{2(T-\tau)}}}(1-x^2)^{(d-3)/2}dxd\tau < \half~.
\]
\end{lemma}

As before, we are also able to bound from below the gap between
$C_d(T)$ and $\half$, given a weak condition on $T$.
For this, we interchange the order of integration and apply the change of
variable $\lambda=\frac{1-x}{2(T-\tau)}$, giving
\[
C_d(T)=\int_{-1}^1
\frac{\omega_{d-1}}{2^{d+1}\pi^{d/2}}\frac{(1+x)^{(d-3)/2}}{\sqrt{1-x}}
\left(2^{d/2}\int_{\frac{1-x}{2T}}^\infty\lambda^{d/2}e^{-\lambda}d\lambda\right)dx
\]
and
\[
\half -C_d(T)=
\frac{1}{2^{d/2}\sqrt{\pi}\Gamma\left(\frac{d-1}{2}\right)}
\int_{-1}^1
\frac{(1+x)^{(d-3)/2}}{\sqrt{1-x}}
\left(\int_0^{\frac{1-x}{2T}}\lambda^{d/2}e^{-\lambda}d\lambda\right)dx.
\]
Assume now $T\geq 1$. Then for $x\in [-1,1]$, $\frac{1-x}{2T}\leq 1$.
Thus, $e^{-\lambda}\geq e^{-1}$ for $\lambda\in[0,\frac{1-x}{2T}]$ and
\be
\ba
\half -C_d(T)&\geq
\frac{1}{2^{d/2}\sqrt{\pi}\Gamma\left(\frac{d-1}{2}\right)}
\int_{-1}^1
\frac{(1+x)^{(d-3)/2}}{\sqrt{1-x}}
\left(\frac{1}{e}\int_0^{\frac{1-x}{2T}}\lambda^{d/2}d\lambda\right)dx\\
&=\frac{1}{ed2^{d-1}\sqrt{\pi}\Gamma\left(\frac{d-1}{2}\right)T^{\frac{d}{2}}}
\int_{-1}^1
(1-x^2)^{(d-3)/2}(1-x)dx\\
&=\frac{2}{ed2^{d-1}\sqrt{\pi}\Gamma\left(\frac{d-1}{2}\right)T^{\frac{d}{2}}}
\int_{0}^1
(1-x^2)^{(d-3)/2}dx\\
&=\frac{2}{ed2^{d-1}\sqrt{\pi}\Gamma\left(\frac{d-1}{2}\right)T^{\frac{d}{2}}}
\int_0^{\frac{\pi}{2}}\cos^{d-2}(\theta)d\theta\\
&=\frac{1}{ed2^{d-1}\Gamma\left(\frac{d}{2}\right)
  T^{\frac{d}{2}}},
\ea
\label{cdbound}
\ee
where the last equality follows from an integral identity in
\cite[\S3.62]{Table}.

Armed with this polynomial control of the gap, and
following the same reasoning as used to show~\eqref{2dbound}, we obtain
the following theorem regarding the stability of the forward
Euler scheme in higher dimensions.
\begin{theorem} Fix $d>2$, and $T\ge 1$.
  For all $n=0,1,\ldots$ and $m=1,\ldots,a_{n,d}$,
\be
\|\bsigma^{nm}\|\leq \frac{1}{\half -C_d(T)}\|\mbf{f}^{nm}\|
\leq ed2^{d-1}\Gamma\left(\frac{d}{2}\right) T^{\frac{d}{2}}\|\mbf{f}^{nm}\|
\label{ndbound}
\ee
for all $N$, $\dt$ such that $N\dt\leq T$.
That is, when $\Gamma$ is the unit sphere $\sd$, the forward Euler scheme
  for solving the second kind Volterra integral equation \eqref{bie}
  is unconditionally stable on any finite time interval $[0,T]$.
\end{theorem}

\begin{remark}
  When $d=2$, $H_n$ is spanned by $e^{in\theta}$ and $e^{-in\theta}$.
  The decomposition of $L^2(S^1)$ into spherical harmonics is the usual
  Fourier series expansion. And if we identify $P_{n,1}(x)$ with the Chebyshev
  polynomials $T_n(x)$, then all calculations in this subsection are valid
  for $d=2$. We instead presented the analysis in two dimensions using the usual
  Fourier series for the reader's convenience. 
\end{remark}

\begin{remark}
  It is easy to see that the bound \eqref{ndbound} actually
  also includes the cases of $d=1$ and $d=2$ proved earlier, thus
  holds for all $d\ge1$.
\end{remark}
  
%
%
\subsection{The Neumann problem on the unit ball}
For the Neumann condition \eqref{neumannbc}, we represent $u^{(B)}$ as
the single layer potential $\cs[\sigma]$. The jump relation \eqref{snjump}
leads to the second kind Volterra integral equation
\be\label{neumannbie}
\left(\half+\cs_{\bnu} \right)[\sigma](\bx,t)= 
\tilde{g}(\bx,t), \quad (\bx,t)\in \Gamma\times [0,T],
\ee
where $\cs_{\bnu}$ indicates the normal derivative
of the single layer with respect to the target point, restricted to $\Gamma$, interpreted in a principal value sense,
as in section~\ref{sec:potentials}.
In \eqref{neumannbie} the right hand side is the corrected data
\[ \tilde{g}(\bx,t) := g(\bx,t) - \frac{\partial u^{(F)}(\bx,t)}{\partial \bnu}~,
\qquad \bx\in\Gamma~.
\]
On the unit sphere $\sd$, straightforward calculation shows that the kernel
of the double layer potential $\cd$ is exactly the same as that of $\cs_\bnu$.
Thus, the forward Euler scheme for \eqref{neumannbie} leads to the identical
marching matrix except a sign change in diagonal entries. Since we prove the
bound \eqref{ndbound} by bounding the spectral radius of the marching matrix
excluding the diagonal part, we observe that \eqref{ndbound} holds
as well for \eqref{neumannbie} with $f$ replaced by $\tilde{g}$.  This leads
to the unconditional stability of the forward Euler scheme for the Neumann
problem on the unit ball, for all $d\ge1$.
%
%


\section{The Robin problem on the half space} \label{sec:robin}

For the Robin boundary condition \eqref{robinbc}, we also represent
$u^{(B)}$ via a single layer potential $\cs[\sigma]$. 
The jump relation \eqref{snjump} leads to the 2nd-kind Volterra
equation
\be\label{robinbie0}
\left(\half+\cs_{\bnu}+\kappa \cs \right)[\sigma](\bx,t)= 
\tilde{h}(\bx,t), \quad (\bx,t)\in \Gamma\times [0,T],
\ee
with corrected Robin data
\be
\tilde{h}(\bx,t) := h(\bx,t)  - \frac{\partial u^{(F)}(\bx,t)}{\partial \bnu}
- \kappa u^{(F)}(\bx,t)~.
\label{robcorr}
\ee
When $D=\mathbb{R}^d_+$, where $\Gamma$ is naturally identified as
$\mathbb{R}^{d-1} \subset \mathbb{R}^d$, the kernel of $\cs_\bnu$ is identically zero
due to the fact that $(\bxy) \cdot \bnu_\bx=0$. Thus, \eqref{robinbie0}
reduces to
\be\label{robinbie}
\left(\half+\kappa \cs \right)[\sigma](\bx,t)= 
\tilde{h}(\bx,t), \quad (\bx,t)\in \mathbb{R}^{d-1}\times [0,T].
\ee
Here we assume that $\tilde{h}$ is sufficiently smooth and decays sufficiently
fast at infinity so that the problem is well posed.

\subsection{The Robin problem in one dimension}
\label{s:robin1d}

In one dimension, the boundary $\Gamma$ of the half line consists of
a single point $x=0$. The integral equation \eqref{robinbie}
reduces to the Abel integral equation
(multiplying both sides by two, and denoting the right-hand side by $f$ instead): 
\begin{equation}
\sigma(t)+\frac{\kappa}{\sqrt{\pi}}\int_0^t
\frac{\sigma(\tau)}{\sqrt{t-\tau}}d\tau \;=\; 2f(t).
\label{robinbie1d}
\end{equation}
Before discretizing, we show stability of the continuous problem for $\kappa>0$.
The Riemann--Liouville fractional integral operator
$\mathcal{R}_\alpha$ is defined by the formula
\[
\mathcal{R}_\alpha[g](t)=\frac{1}{\Gamma(\alpha)}
\int_0^t\frac{g(\tau)}{(t-\tau)^{1-\alpha}}d\tau, \qquad \alpha\in (0,1),
\]
where $\Gamma(\alpha)$ is
the gamma function~\eqref{gammadef}. Thus, the integral
operator on the left side of \eqref{robinbie1d} is simply
$\frac{\Gamma(1/2) \kappa}{\sqrt{\pi}} \mathcal{R}_{1/2}=
 \kappa \mathcal{R}_{1/2}$.
 For all real functions $g$,
$\mathcal{R}_\alpha$ satisfies
the positivity property \cite[Lemma 3.1]{mustapha2014}
\begin{equation}
    \int_0^Tg(t)\mathcal{R}_\alpha[g](t)dt \ge 0~.
    \label{nonneg}
\end{equation}
Taking the inner product of \eqref{robinbie1d} with
$\sigma$ over a fixed interval $[0,T]$, and using \eqref{nonneg}, gives
\[
\|\sigma\|_{L^2([0,T])}^2 \le 2 (\sigma,f) \le 2 \|\sigma\|_{L^2([0,T])}
\|f\|_{L^2([0,T])}
\]
where Cauchy--Schwartz was used in the last step.
So on any finite interval $[0,T]$ this gives
the continuous version of the $L^2$ stability bound
\[
\|\sigma\|\le 2\|f\|~.
\]

We now proceed to discretization. Recall that
the forward Euler scheme uses a piecewise constant approximation
$\sigma(t) \approx \sigma_m:=\sigma(t_m)$ on $[t_m,t_{m+1})$
on the uniform grid $t_m = m\dt$.
Then performing the integrals exactly in \eqref{robinbie1d} gives
the explicit marching rule
\be\label{mun}
\sigma_n = 2f_n - \sum_{m=0}^{n-1} v_{n-m}\sigma_m, 
 \qquad n=1,\ldots,N
\ee
with the lower-triangular Toeplitz matrix weights
\be  \label{wdef}
v_j = 2\sqrt{h}(\sqrt{j}-\sqrt{j-1}) = 
\frac{2\sqrt{h}}{\sqrt{j} + \sqrt{j-1}}~, \qquad j=1,2,\ldots~,
\ee
and where $f_n := f(t_n)$,
and $h := \kappa^2 \, \dt/\pi$.
For smooth solutions $\sigma \in C^1([0,T])$,
this rule can be proved to be first-order accurate
by combining compactness of the integral operator,
C\'ea's lemma, and noting that
the piecewise constant approximant has error $\bigO(\dt)$
(see \cite[Sec.~13.1--3]{kress2014}).

For initialization, as before we set $\sigma_0 = f_0 = 0$.
Let us define the vectors $\bsigma$ and $\mbf{f}$ by
$\{\sigma_n\}_{n=0}^N,  
\{f_n\}_{n=0}^N  \in \RR^{N+1}$, respectively.
Using this notation, \eqref{mun} takes the form of
the lower-triangular Toeplitz linear system
\be
(I + V)\bsigma = 2 \mbf{f}~,
\label{munv}
\ee
where $V\in\RR^{(N+1)\times(N+1)}$
has elements $v_{n,m} = v_{n-m}$ for $n>m$, and $v_{n,m} = 0$ otherwise.
Here, $v_n$ is defined in \eqref{wdef} with $h = \kappa^2 \, \dt/\pi$.

There is a substantial literature on the numerical analysis and stability of 
Volterra equations in the one-dimensional setting.
For a discussion of convergence theory and step-size control, see
\cite{baker2000,jonesmckee1982} and the monograph \cite{Brunner2004}.
Much work on stability has been devoted to an analysis of the model problem
\[ y(t) + \int_0^t [\lambda_0 + \lambda_1 (t-\tau)] y(\tau) \, d\tau = f(t), \]
or to problems with a continuous kernel 
\cite{jonesmckee1982,messinavecchio}.
In \cite{lubich1983ima}, a more relevant stability result is obtained for systems
of the form \eqref{munv}, but assuming that the 
sequence $\{v_j\}$ is in $l^1$, which is not the
case here.

For previous work on Abel-type equations with singular kernels, 
we refer the reader to
\cite{eggermont,lubich1983mcom,lubich1985mcom,vogeli2018acom}. 
These papers, however, are mostly concerned with implicit marching schemes.
An exception is Lubich's 1986 paper \cite{lubich1986ima}, which does a
careful stability analysis for a variety of schemes and makes clear the connection between
completely monotonic sequences and stability. An interesting
result from that paper is Corollary 2.2, which states that 
``the stability region of an explicit convolution quadrature $\dots$ is bounded." 
Theorem \ref{1dbound} below, which is consistent with Lubich's
result, gives a precise value for the time step restriction.
It also guarantees that $\sigma$ decays once the right-hand side $f$ has switched off.

\begin{theorem}\label{1dbound}
  There is a constant $0<c< 3-\sqrt{2}$ such that, for any $N$ and any
  $\mbf{f} \in \RR^{N+1}$,
  the solution to \eqref{munv} obeys
  \be
  \|\bsigma\| \le \frac{2}{1-c\sqrt{h}} \|\mbf{f}\| ~,
  \label{bnd}
  \ee
  where $\|.\|$ denotes the $l^2$-norm.
That is, the marching scheme \eqref{mun}
is stable for $h < 0.39 < (1/c)^2$ or
$\dt < \pi/(c^2 \kappa^2)$,  where $\kappa$ is the heat transfer coefficient.
\end{theorem}

\begin{proof}
  We first show that there exists a constant $c>0$ such that
  \be
  \bsigma^T V\bsigma \ge -c \sqrt{h} \|\bsigma\|^2 \qquad 
  \mbox{for any} \quad \bsigma \in \RR^{N+1},
  \label{below}
  \ee
  \ie that the smallest eigenvalue of $V$ is bounded from below.
  Writing $\sqrt{h} W_{N+1} := \half (V+V^T)$ as the scaled symmetric part
  of $V$, note that
  $\bsigma^T V \bsigma = \sqrt{h} \bsigma^T W_{N+1} \bsigma$, and that $W_{N+1}$ is
  independent of the time-step.
  Note that $W_{N+1}$ is the $(N+1)\times(N+1)$ upper left principal submatrix
  of the infinite symmetric
  Toeplitz matrix $T_v$, defined by the sequence $0, v_1, v_2, \ldots$
  with 
\[
  v_j = \frac{1}{\sqrt{j}+\sqrt{j-1}}=
  \sqrt{j}-\sqrt{j-1},\qquad j\in\mathbb{N}.
\]
  It is straightforward to check that the sequence $\{v_j\}_{j\in\mathbb{N}}$
  is convex and that $\lim_{j\rightarrow \infty}v_j=0$. By Theorem~\ref{lowerbound}
  and Remark~\ref{finitetoeplitz}, we have
  \[
  \bsigma^T W_{N+1}\bsigma \geq (v_2-2v_1) \|\bsigma\|^2.
  \]
  That is, \eqref{below} holds if $c=2v_1-v_2=3-\sqrt{2}$.
  To complete the proof, take the inner product of \eqref{munv}
  with $\bsigma$ to get
\[
  \|\bsigma\|^2 + \bsigma^T V \bsigma = 2\bsigma^T\mbf{f} \, .
\]
  Applying \eqref{below} to the left-hand side and the Cauchy--Schwartz inequality
  to the right-hand side, we have
  \[
  (1 - c\sqrt{h})\|\bsigma\|^2 \le 2 \|\bsigma\| \|\mbf{f}\| \, ,
  \]
  from which \eqref{bnd} follows for any $\bsigma\neq \mbf{0}$. It holds trivially
  when $\bsigma = \mbf{0}$.
\end{proof}

\begin{remark}
The above proof gives $c=3-\sqrt{2} \approx 1.5858$. 
By numerically computing the smallest eigenvalue of successively larger Toeplitz matrices
$V$, or, better, by evaluating $v(\pi)=2\sum_{j>0} (-1)^{j-1} v_j$,
one can obtain an optimal estimate of $c \approx 1.52041925043874$.
We omit the details of this computation and mention it only to illustrate that
the explicit bound is within about 4\% of the optimal one.
\label{robinremark1}
\end{remark}

\begin{remark}
  With unit diffusion constant, the transfer coefficient $\kappa$ has units $(\mbox{length})^{-1}$. Thus our time-step condition
  $\dt < \pi/(c\kappa)^2$ is proportional to the square of the physical length $1/\kappa$.
  Although reminiscent of the explicit finite-difference stability condition $\dt < c \dx^2$, our stability condition is, by contrast, independent of any spatial discretization. (Indeed, in practice
  the only spatial discretization needed would be quadrature to evaluate
  \eqref{uF} to get $f(t)$, as in \eqref{robcorr}. With $f(t)$ computed, there is no spatial variable left to discretize.)
\end{remark}
\subsection{The Robin problem in higher dimensions}
In higher dimensions $d\ge 2$, the boundary $\Gamma$ of the half space $\mathbb{R}^d_+$
can be identified as $\mathbb{R}^{d-1}$ by natural embedding.
The integral equation \eqref{robinbie} is rewritten as
\be
\sigma(\bx,t)+\frac{\kappa}{\sqrt{\pi}}\int_0^t
\frac{1}{\sqrt{t-\tau}}\int_{\mathbb{R}^{d-1}}
\frac{1}{(4\pi (t-\tau))^{(d-1)/2}}e^{-\frac{|\bx-\by|^2}{4(t-\tau)}}
\sigma(\by,\tau)d\by d\tau \;=\; 2f(\bx,t),
\quad \bx,\by \in \mathbb{R}^{d-1}.
\label{robinbiend}
\ee
Again we have multiplied both sides by two and denote the right hand side by $f$.
We observe that the kernel inside the spatial integral on the left side of
\eqref{robinbiend} is exactly the heat kernel in $\mathbb{R}^{d-1}$. It is well known
that the Fourier transform of the heat kernel $G(\bx,t)$
in $\mathbb{R}^{d-1}$
is simply $e^{-|\bxi|^2t}$,
in terms of the Fourier variable $\bxi\in\mathbb{R}^{d-1}$.
Using this fact and that the convolution in physical space becomes pointwise multiplication in frequency,
and taking the Fourier transform in $\mathbb{R}^{d-1}$ of
both sides of \eqref{robinbiend},
we obtain
\be
\hat{\sigma}(\bxi,t)+\frac{\kappa}{\sqrt{\pi}}\int_0^t
\frac{e^{-|\bxi|^2(t-\tau)}}{\sqrt{t-\tau}}
\hat{\sigma}(\bxi,\tau) d\tau \;=\; 2\hat{f}(\bxi,t),
\quad \bxi \in \mathbb{R}^{d-1}~.
\label{robinbiendf}
\ee
Note that in the special case $\bxi=0$ this recovers
\eqref{robinbie1d}.

Fixing $\bxi$, we proceed similarly as in the one-dimensional case.
That is,
we approximate $\hat{\sigma}(\bxi,t)$ by a constant
$\hat{\sigma}_m(\bxi):=\hat{\sigma}_m(\bxi,t_m)$ on $[t_m,t_{m+1})$
with $t_m = m\dt$, and perform the integrals exactly.
Let us define the vectors $\bsigmahat$ and $\bfhat$ by
$\{\hat{\sigma}_n(\bxi)\}_{n=0}^N,  
\{\hat{f}_n(\bxi)\}_{n=0}^N  \in \RR^{N+1}$, respectively.
Using this notation, the forward Euler scheme for \eqref{robinbiendf} takes the form of
the lower-triangular Toeplitz linear system
\be
(I + \Vhat)\bsigmahat = 2 \bfhat~,
\label{sigmahatfe}
\ee
where $\Vhat\in\RR^{(N+1)\times(N+1)}$
has elements $v_{n,m}(\bxi) = v_{n-m}(\bxi)$ for $n>m$, and $v_{n,m} = 0$ otherwise.
Here, $v_n$ is defined by
\be
v_n(\bxi) = 2\sqrt{h} \frac{1}{2\sqrt{\dt}}
\int_0^{\dt} \frac{e^{-|\bxi|^2(n\dt-\tau)}}{\sqrt{n\dt-\tau}}d\tau
\label{vnnd}
\ee
with, as before, $h = \kappa^2 \, \dt/\pi$.

\begin{lemma}
For any $\dt>0$ and any fixed $\bxi$, the sequence $\{v_n(\bxi)\}_{n\in \mathbb{N}}$ is convex.
\label{vnconvex}
\end{lemma}
\begin{proof}
  Let $x=|\bxi|^2\dt$. Applying the change of variables $u=n-\tau/\dt$ on the integral
  in \eqref{vnnd} leads to
  \be
  v_n(\bxi)=\sqrt{h} \int_{n-1}^n \frac{e^{-xu}}{\sqrt{u}}du.
  \label{vnnd1}
  \ee
  Thus, in order to show that $\{v_n(\bxi)\}_{n\in \mathbb{N}}$ is a convex sequence,
  we only need to show that the function
  \[
  g(t)=\int_{t-1}^t \frac{e^{-xu}}{\sqrt{u}}du
  \]
  is convex for $t\geq 1$. Here $x\ge 0$ is a fixed parameter.
  Differentiating $g(t)$ twice leads to
  \[
  g''(t) = p'(t)-p'(t-1)
  \]
  with
  \[
  p(t)=\frac{e^{-xt}}{\sqrt{t}}.
  \]
  Now
  \[
  p''(t) = \frac{3\,{\mathrm{e}}^{-t\,x}}{4\,t^{5/2}}+\frac{x\,{\mathrm{e}}^{-t\,x}}{t^{3/2}}+\frac{x^2\,{\mathrm{e}}^{-t\,x}}{\sqrt{t}},
  \]
  which is positive for any $x\ge 0$ and $t>0$.
  This shows that $p'(t)$ is monotonically increasing for $t>0$.
  Thus, $p'(t)>p'(t-1)$ for $t\ge 1$, and $g''(t)>0$ for $t\ge 1$, completing the proof.
\end{proof}

\begin{lemma}
  For any $\bxi \in \mathbb{R}^{d-1}$,
  \be
  v_2(\bxi)-2v_1(\bxi) \ge v_2(0)-2v_1(0) = -2(3-\sqrt{2})\sqrt{h}.
  \label{v0bound}
  \ee
\end{lemma}

\begin{proof}
  Using the expression \eqref{vnnd1}, we only need to show that
  \[f(x):=\int_1^2 \frac{e^{-xu}}{\sqrt{u}}du-2\int_0^1 \frac{e^{-xu}}{\sqrt{u}}du\ge f(0)\]
  for $x\ge 0$.
  For this, we calculate
  \[
  \ba
  f'(x)&=2\int_0^1 \sqrt{u}e^{-xu}du-\int_1^2 \sqrt{u}e^{-xu}du\\
  &=3\int_0^1 \sqrt{u}e^{-xu}du-\int_0^2 \sqrt{u}e^{-xu}du\\
  &=3\int_0^1 \sqrt{u}e^{-xu}du-2\sqrt{2}\int_0^1 \sqrt{u}e^{-2xu}du\\
  &>2\sqrt{2}\int_0^1 \sqrt{u}e^{-xu}(1-e^{-xu})du\\
  &\ge 0, \qquad \mbox{for} \quad x\ge 0.
  \ea
  \]
  That is, $f$ is monotonically increasing for $x\ge 0$, completing the proof.
\end{proof}

Lemma \ref{vnconvex} together with Theorem~\ref{lowerbound}
and Remark~\ref{finitetoeplitz} leads to
\be
\bsigmahat^T \Vhat\bsigmahat \ge \half (v_2(\bxi)-2v_1(\bxi)) \|\bsigmahat\|^2
\qquad 
\mbox{for any} \quad \bxi \in \mathbb{R}^{d-1}.
\label{sigmahatbound0}
\ee
Combining the above estimate with \eqref{v0bound}, we obtain
\be
\bsigmahat^T \Vhat\bsigmahat \ge 
-c\sqrt{h}\|\bsigmahat\|^2
\qquad 
\mbox{for any} \quad \bxi \in \mathbb{R}^{d-1},
\label{sigmahatbound}
\ee
where
\be
c=3-\sqrt{2}.
\ee

An argument similar to that in the proof of Theorem \ref{2dthm1}
then gives
\be
\|\bsigmahat\| \le \frac{2}{1-c\sqrt{h}} \|\bfhat\|
\qquad 
\mbox{for any}
\quad \bxi \in \mathbb{R}^{d-1}.
\label{bndnd0}
\ee
Taking the $L^2$-norm in Fourier space and then applying the Plancherel
theorem, we have
\be
\|\bsigma\| \le \frac{2}{1-c\sqrt{h}} \|\mbf{f}\|.
\label{bndnd}
\ee
That is, we obtain exactly the same bound \eqref{bnd} as in one dimension, which
shows that the forward Euler scheme is stable for \eqref{robinbiend}
 if $\dt < \pi/(c^2 \kappa^2)$,
where $\kappa$ is the heat transfer coefficient.

\begin{remark}
In the limit $\kappa \rightarrow 0$, the 
scheme is unconditionally stable. This is to be expected, since when $\kappa=0$,
the Robin boundary condition becomes a Neumann condition and the integral
equation \eqref{robinbie} yields the analytic solution $\sigma(\bx,t) = 2\tilde{h}(\bx,t)$.
\end{remark}
%
%

\section{The Dirichlet problem on an arbitrary smooth convex domain} \label{sec:convex}

We now study the stability property of the forward Euler scheme \eqref{biemarch}
for the Dirichlet problem on an arbitrary $C^1$ convex domain, i.e., the boundary
integral equation \eqref{bie}.

We first establish a connection between the heat kernel and the Laplace kernel.
The Green's function for the Laplace 
equation in $\mathbb{R}^d$ is
\[
G_{\rm L}(\bx,\by)=\left\{\begin{array}{cc}
-\frac{1}{2\pi}\ln |\bx-\by|, \quad & d=2,\\
\frac{1}{(d-2)\omega_d} \frac{1}{|\bx-\by|^{d-2}},  \quad & d\geq 3,
\end{array}\right.
\]
where
\be\label{spherearea}
\omega_d=\frac{2\pi^{d/2}}{\Gamma(d/2)}
\ee
is the area of the unit sphere $\sd\subset \mathbb{R}^d$.
Here $\Gamma$ is the gamma function defined by the formula
\be\label{gammadef}
\Gamma(z)=\int_0^\infty x^{z-1}e^{-x}dx.
\ee
The kernel of the
Laplace double layer potential operator
is given by
\be\label{ldkernel}
\ldk
=\frac{\Gamma(d/2)}{2\pi^{d/2}} \frac{(\bx-\by)\cdot\bnu(\by)}{|\bx-\by|^d}.
\ee
It is well known to satisfy Gauss' Lemma~\cite{kress2014}:
\be\label{gausslemma}
\int_\Gamma\ldk dS(\by)=-\half~,
\quad \bx\in \Gamma.
\ee

\begin{lemma}
\be\label{hlconnection}
\lim_{t\rightarrow \infty}\int_0^t\dk d\tau = \ldk.
\ee
\end{lemma}
\begin{proof}
By \eqref{dkernel}, we have
\[
\int_0^t\dk d\tau = 
\frac{(\bxy)\cdot\bnu(\by)}{2^{d+1}\pi^{d/2}}\int_0^t\frac{1}{(t-\tau)^{1+d/2}}e^{-\frac{|\bx-\by|^2}{4(t-\tau)}}d\tau.
\]
The change of variables $\lambda=\frac{|\bx-\by|^2}{4(t-\tau)}$
leads to
\be
\int_0^t\dk d\tau=\frac{(\bxy)\cdot\bnu(\by)}{2\pi^{d/2}|\bx-\by|^d} \int_{\frac{|\bx-\by|^2}{4t}}^\infty
\lambda^{\frac{d}{2}-1}e^{-\lambda}d\lambda.
\label{hlconnection1}
\ee
Taking the limit $t\rightarrow \infty$
and using the definition of the gamma
function~\eqref{gammadef},
we obtain \eqref{hlconnection}.
\end{proof}

The following provides the key ingredient for the stability of the forward
Euler scheme in an arbitrary smooth convex domain.
\begin{lemma}\label{convexestimates}
  Suppose that $D$ is a $C^1$ convex domain. Then
  \be\label{nonpositivity}
  \dk\leq 0, \quad \ldk\leq 0, \quad \bx, \by \in \Gamma,
  \ee
and
\be\label{dlplimit}
\lim_{t\rightarrow \infty}
  \int_0^t\int_\Gamma \dk dS(\by)d\tau= -\rhalf, \quad \bx\in \Gamma.
  \ee
 For $t\in (0,\infty)$, define
  \be
  \label{cTdef}
  C(t)=\left\|\int_0^t\int_\Gamma \left|\dk\right| dS(\by)d\tau\right\|_\infty.
  \ee
Then $C(t)$ is a monotonic increasing function of $t$ and 
  \be\label{convexbound}
  C(t) < \rhalf.
  \ee
\end{lemma}
\begin{proof}
  \eqref{nonpositivity} follows from the expressions~\eqref{dkernel} and
  \eqref{ldkernel} and the fact that $\xdy\leq 0$ for $\bx, \by \in \Gamma$
  when $D$ is convex due to the convex separation theorem~\cite{boyd2004}.
  \eqref{dlplimit} follows from \eqref{gausslemma} and \eqref{hlconnection}.
  The monotonic increasing property of $C(t)$ follows from \eqref{hlconnection1}
  and the fact that the integrand is of the same sign everywhere by \eqref{nonpositivity}.
  Finally, \eqref{convexbound} is a simple consequence of \eqref{nonpositivity}
  and \eqref{dlplimit}.
\end{proof}

Recall that the forward Euler scheme \eqref{biemarch} for the Dirichlet problem
is
\begin{align}
\label{biemarch1}
\half \sigma(\bx, n \dt) &=   \sum_{j = 0}^{n-1}
 \int_{j\dt}^{(j+1)\dt} \int_\Gamma
\frac{\partial G(\bx-\by,n\dt -\tau)}{\partial \bnu(\by)}
\sigma(\by,j\dt)ds(\by)d\tau \\
& \qquad- \; f(\bx, n\dt).  \nonumber
\end{align}
Here we have dropped the tilde from $f$ again.

\begin{theorem}
  Let $D\subset\RR^d$ be a bounded, convex domain with
  $C^1$-boundary. Fix $T>0$. The solution $\sigma$
  to~\eqref{biemarch1} satisfies \be \|\sigma\|_\infty \;\leq\;
  \frac{1}{\half-C(T)} \| f\|_\infty
  \label{convexstability}
  \ee
  for any $N$, $\dt$ such that $N\dt \leq T$. Here $C(T)$ is defined in \eqref{cTdef},
  $\|\cdot\|_\infty$ denotes
  the $L^\infty$ norm in space and the $l^\infty$ norm in the discrete temporal variable.
  In other words, the forward scheme \eqref{biemarch} is unconditionally stable on $[0,T]$
  for any $T>0$.
\end{theorem}
\begin{proof}
Taking the absolute value on both sides of \eqref{biemarch1}, we have
\be
\ba
\half |\sigma(\bx,n \dt)|&\leq \sum_{j = 0}^{n-1}
\int_{j\dt}^{(j+1)\dt} \int_\Gamma
\left|\frac{\partial G(\bx-\by,n\dt -\tau)}{\partial \bnu(\by)}
\sigma(\by,j\dt)\right|ds(\by)d\tau \\
& \quad+ \; |f(\bx, n\dt)|\\
& \leq    \sum_{j = 0}^{n-1}
\int_{j\dt}^{(j+1)\dt} \|\sigma(\cdot,j\dt)\|_\infty
\int_\Gamma
\left|\frac{\partial G(\bx-\by,n\dt -\tau)}{\partial \bnu(\by)}\right|
ds(\by)d\tau \\
&\quad+ \| f(\cdot, n\dt)\|_\infty\\
&\leq \|\sigma\|_\infty\sum_{j = 0}^{n-1}
\int_{j\dt}^{(j+1)\dt} \int_\Gamma
\left|\frac{\partial G(\bx-\by,n\dt -\tau)}{\partial \bnu(\by)}\right|
ds(\by)d\tau + \| f\|_\infty\\
&=\|\sigma\|_\infty
\int_{0}^{(n-1)\dt} \int_\Gamma
\left|\frac{\partial G(\bx-\by,n\dt -\tau)}{\partial \bnu(\by)}\right|
ds(\by)d\tau + \| f\|_\infty,
\label{maxbound0}
\ea\ee
where the first inequality follows from the triangle inequality, the second one
following from taking the $L^\infty$ norm in the spatial variable
for both $\sigma$ and $f$, and the third one follows from taking the maximum
norm in the discrete temporal variable.
We continue our calculation
\be
\ba
\rhalf  |\sigma(\bx,n \dt)|&\leq \|\sigma\|_\infty
\int_{0}^{n\dt} \int_\Gamma
\left|\frac{\partial G(\bx-\by,n\dt -\tau)}{\partial \bnu(\by)}\right|
ds(\by)d\tau + \| f\|_\infty\\
&\leq C(T) \|\sigma\|_\infty+ \| f\|_\infty.
\ea\ee
Since the above inequality is valid for any $\bx\in\Gamma$ and any $n$ such that
$n\dt\leq T$, its left hand side can be replaced by $\half \|\sigma\|_\infty$,
completing the proof.
\end{proof}

\section{Conclusions and further remarks}
\label{s:conc}

We have analyzed the stability of the forward Euler
scheme for solving the Dirichlet and Neumann problems
for the heat equation in the unit ball, with data specified on the unit sphere
$\sd\subset\mathbb{R}^d$, using second-kind Volterra
time-domain boundary integral equations.
While finite difference
methods require that the Courant number $\dt/(\Delta x)^2$ be $\mathcal{O}(1)$,
we have shown that integral equation methods can be both explicit and unconditionally
stable for any fixed final time $T$.

We have also studied the Robin problem on the half space in all dimensions and shown
that stability of the forward Euler scheme follows if
$\dt < \frac{\pi}{c^2\kappa^2}$, where $c=3-\sqrt{2}$ and $\kappa$ is the heat
transfer coefficient. As pointed out in Remark \ref{robinremark1}, this bound is
very close to the optimal bound where $c\approx 1.52041925043874$.

A critical element in the proof of unconditional stability
of the forward Euler scheme is the pointwise non-positivity
of the double layer heat kernel on the unit sphere $\sd$,
a property which extends to any convex domain.
Combining this with the elementary fact that a unit double-layer
density generates a surface potential approaching $-\half$
enabled us to extend this stability result to arbitrary
smooth convex domains, in the Dirichlet case and the $L^\infty$-norm.

A key ingredient in the Robin proofs was a bound on the
smallest eigenvalue of real symmetric Toeplitz matrices via the convexity
of the associated  sequence.
This may be of independent interest in signal processing applications.
Another ingredient for the proofs
was a tight rational function bound for the ratio of modified Bessel functions
of the first kind with large positive real argument, which may be of interest in its own right.
A detailed analysis combining these ingredients
showed that in the Dirichlet disc ($d=2$), the density is
bounded in norm by the data, uniformly in time, so long as $\dt\leq 1$.

While this paper is purely analytic, we note that the numerical experiments in 
\cite{wang2017nyu} are consistent with the theory presented here.
More detailed experiments will be reported in a forthcoming paper
\cite{fullheatsolver} that considers the full initial-boundary value problem including
forcing terms.

Some other questions arise naturally from our study. First, for the
Dirichlet problem on the unit ball in higher dimensions, one may ask
whether the scheme is stable for all time given some mild constraint
on $\dt$. Second, one may ask about the stability analysis of the
Robin problem on the unit ball in all dimensions. Third, it is natural
to inquire about the stability of other explicit time marching schemes
such as Adams--Bashforth multistep methods or explicit Runge--Kutta
methods. Fourth, it would be interesting to see if the convexity
assumption could be relaxed, and stability proved for arbitrary,
sufficiently smooth domains. Integral equation methods become
difficult to analyze when the boundary of the domain is not at least
$C^1.$ We are currently investigating these issues and will
report our findings in the future.

\section*{Acknowledgments}
We are grateful for a discussion with Marcus Webb of KU Leuven on the Fourier
series approach. S. Jiang was supported by NSF under grant DMS-1720405
and by the Flatiron Institute, a division of the Simons Foundation. We are also grateful to the anonymous referees for a careful reading of our paper and many useful suggestions for improvement.

\appendix
\section{Properties of the modified Bessel functions of the first kind}
The modified Bessel function of the first kind $I_\nu(x)$
is defined by the formula~\cite[Chapter 10]{nisthandbook}
\[
I_\nu(z)=\left(\half z\right)^\nu\sum_{k=0}^\infty
\frac{\left(\half z\right)^{2k}}{k!\Gamma(\nu+k+1)}.
\]
It satisfies the recurrence relations~\cite[\S10.29.2]{nisthandbook}
\be\label{inrec}
I'_\nu(z)=I_{\nu-1}(z)-\frac{\nu}{z}I_\nu(z),\ \ 
I'_\nu(z)=I_{\nu+1}(z)+\frac{\nu}{z}I_\nu(z).
\ee
When $\nu$ is fixed and $x\rightarrow \infty$~\cite[\S10.30.4]{nisthandbook},
\[
I_\nu(x)\sim \frac{e^x}{\sqrt{2\pi x}}, \quad x\in \mathbb{R}.
\]
When $\nu$ is an integer $n$, $I_n$ admits the integral representation
~\cite[\S10.32.3]{nisthandbook}
\be\label{inrep}
I_n(z)=\frac{1}{\pi}\int_0^\pi e^{z\cos\theta}\cos(n\theta)d\theta.
\ee
The following results can be found in \cite{yang2017mjm}.
\begin{lemma}\label{wnubound}
  Let $W_\nu(x)=\frac{xI_\nu(x)}{I_{\nu+1}(x)}$ and
  $S_{p,\nu}=W^2_\nu(x)-2pW_\nu(x)-x^2$. Then $S_{\nu,\nu-1}(x)$ is
  monotonically decreasing from $0$ to $-\infty$ on $(0,\infty)$ for $\nu>1/2$,
  \be\label{wnubound1}
  \nu-\half +
  \sqrt{x^2+\nu^2-\frac{1}{4}}
  \leq W_{\nu-1}(x) \leq 
  \nu-\half +
  \sqrt{x^2+\left(\nu+\half \right)^2},
  \ee
and 
  \be\label{wnubound2}
  \nu-1+
  \sqrt{x^2+(\nu+1)^2}
  \leq W_{\nu-1}(x)
  \ee
  for $\nu\geq\half $,
  with $x\in(0,\infty)$.
\end{lemma}

\begin{lemma}\label{inbound1}
Let $n$ be a positive integer. Then
\begin{enumerate}[label=(\alph*)]
    \item There is only one zero $r_n$ for the
    equation 
    \[
    \frac{I'_n(x)}{I_n(x)}=1
    \]
    on $(0,+\infty)$. Furthermore,
    \be\label{rnbound}
    \max(n^2-\half ,\frac{n^2}{2}+n)\leq r_n \leq n^2+n.
    \ee
    \item The function $e^{-x}I_n(x)$ increases monotonically on $[0,r_n]$ 
    and decreases monotonically on $[r_n,+\infty)$.
    \item The maximum value of $e^{-x}I_n(x)$ on $[0,\infty)$ satisfies
    \be\label{inbound2}
    \max_{[0,+\infty)}e^{-x}I_n(x)<\frac{1}{2n+1}.
    \ee
\end{enumerate}
\end{lemma}

\begin{proof}
\begin{enumerate}[label=(\alph*)]
\item
  Using the recurrence~\eqref{inrec}, we have
  \[
  W_{n-1}(x)=x\frac{I'_n(x)}{I_n(x)}+n.
  \]
  Thus,
  \[
  S_{n,n-1}(x)=x^2\left(\frac{I'_n(x)}{I_n(x)}\right)^2-x^2-n^2.
  \]
  When $\frac{I'_n(x)}{I_n(x)}=1$,  $S_{n,n-1}=-n^2$.
  By the monotonicity and the range of $S_{n,n-1}(x)$, $S_{n,n-1}$ takes the value
  $-n^2$ at only one point and we denote that point by $r_n$.

  Substituting $x=r_n$ into \eqref{wnubound1} and \eqref{wnubound2} with
  $\frac{I'_n(r_n)}{I_n(r_n)}=1$ and simplifying the resulting expressions,
  we obtain \eqref{rnbound}.
\item We have
  \[
  \frac{d}{dx}(e^{-x}I_n(x))=e^{-x}I_n(x) \left(\frac{I'_n(x)}{I_n(x)}-1\right).
  \]
  Using \eqref{wnubound1}, it follows that
  $\frac{I'_n(x)}{I_n(x)}>1$ for $x<n^2-\half $
  and $\frac{I'_n(x)}{I_n(x)}<1$ for $x>n^2+n$. Combing these facts with (a),
  we have $\frac{I'_n(x)}{I_n(x)}>1$ for $x<r_n$
  and $\frac{I'_n(x)}{I_n(x)}<1$ for $x>r_n$. That is,
  $\frac{d}{dx}(e^{-x}I_n(x))>0$ for $x<r_n$
  and $\frac{d}{dx}(e^{-x}I_n(x))<0$ for $x>r_n$, which completes the proof of (b).
\item By the identity \S10.35.5 in \cite{nisthandbook}, we have
  \[
  1=e^{-x}\left(I_0(x)+2\sum_{k=1}^\infty I_k(x)\right).
  \]
  Section 10.37 of \cite{nisthandbook} states that for fixed $x>0$, $I_\nu(x)$
  is positive and decreasing for $0<\nu<\infty$.
  Hence,
  \[
  1>e^{-x}\left(I_n(x)+2\sum_{k=1}^n I_n(x)\right)=(2n+1)e^{-x}I_n(x),
  \]
  which completes the proof.
\end{enumerate}
\end{proof}

The following lemma about differential inequalities can be found
in \cite[Chapter III, \S4]{hartman}. See also \cite{petrovitsch1901}.
\begin{lemma}[Petrovitsch 1901]\label{lem:odeineq}
  Suppose that $f(y,t)$ is continuous in an open domain $D$. Suppose further that
  $y$ is the solution to the Cauchy problem
  \[
  y'(t)=f(y(t),t), \quad y(t_0)=y_0, \quad (y_0,t_0)\in D.
  \]
  
  \begin{enumerate}[label=(\alph*)]
  \item (Increasing $t$). Suppose that $u$ satisfies the inequalities
    \be\label{ucond1}
    \ba
    u'(t)&\geq f(u(t),t), \quad t\in (t_0, t_0+\delta)\, (\delta>0)\\
    u(t_0)&\geq y(t_0).
    \ea
    \ee
    Then
    \be\label{odeineq1}
    u(t)\geq y(t), \quad t\in [t_0,t_0+\delta].
    \ee
    The inequality in \eqref{odeineq1} is reversed if both inequalities in \eqref{ucond1} are reversed.
  \item (Decreasing $t$). Suppose that $u$ satisfies the inequalities
    \be\label{ucond2}
    \ba
    u'(t)&\leq f(u(t),t), \quad t\in (t_0-\delta,t_0)\, (\delta>0)\\
    u(t_0)&\geq y(t_0).
    \ea
    \ee
    Then 
    \be\label{odeineq2}
    u(t)\geq y(t), \quad t\in [t_0-\delta,t_0].
    \ee
    The inequality in \eqref{odeineq2} is reversed if both inequalities in \eqref{ucond2} are reversed.
  \end{enumerate}
\end{lemma}
\begin{lemma}\label{g0sign}
  Let
  \be\label{g0def}
  g_0(x)=(4x-3)I_0(x)-(4x-1)I_1(x).
  \ee
  Then $g_0(x)$ has a unique zero, denoted as $x^\ast$, on $(\frac{3}{4},\infty)$.
  Furthermore, $g_0(x)<0$ on $[\frac{3}{4},x^\ast)$ and $g_0(x)>0$ on $(x^\ast,\infty)$.
\end{lemma}
\begin{proof}
  Let $r_\nu(x)=\frac{I_\nu(x)}{I_{\nu+1}(x)}$.
  In particular,
\[
  r_0(x)=\frac{I_0(x)}{I_1(x)}.
\]
  From \S10.37 of \cite{nisthandbook},
  we know that  $I_\nu(x)$ is positive and increasing on $(0,\infty)$ for fixed
  $\nu(\geq 0)$ and $I_\nu(x)$ is decreasing on $0<\nu<\infty$ for fixed $x$.
  Thus, $r_\nu(x)>1$ on
  $(0,\infty)$ for $\nu\geq0$.
  Let
\[
  l_0(x)=\frac{4x-1}{4x-3}.
\]
  Then it is clear that the sign of $g_0(x)$ is determined by comparing
  $r_0(x)$ with $l_0(x)$.
  First, $\lim_{x\rightarrow \frac{3}{4}^+}l_0(x)=+\infty$ and thus $l_0(x)>r_0(x)$
  as $x\rightarrow \frac{3}{4}^+$.
  Second, the series expansion of $l_0(x)$ and the asymptotic expansion of
  $r_0(x)$ are as follows:
\[
  l_0(x)=1+\frac{1}{2\,x}+\frac{3}{8\,x^2}+\frac{9}{32\,x^3}+\frac{27}{128\,x^4}+\frac{81}{512\,x^5}+O\left(\frac{1}{x^6}\right),
\]
\[
  r_0(x)=1+\frac{1}{2\,x}+\frac{3}{8\,x^2}+\frac{3}{8\,x^3}+\frac{63}{128\,x^4}+\frac{27}{32\,x^5}+O\left(\frac{1}{x^6}\right).
\]
  Hence, $r_0(x)>l_0(x)$ as $x\rightarrow \infty$.
  Combining these two facts, there is at least one point
  $x^\ast\in(\frac{3}{4},\infty)$ where
  $r_0(x^\ast)=l_0(x^\ast)$. Or equivalently,
  \[
  g_0(x^\ast)=0.
  \]
  By the recurrence relations~\eqref{inrec}, $r_\nu$ satisfies the following
  Riccati equation
\[
  r'_\nu(x)=1+\frac{2\nu+1}{x}r_\nu(x)-r^2_\nu(x).
\]
  In particular, for $\nu=0$,
\[
  r'_0(x)=1+\frac{1}{x}r_0(x)-r^2_0(x).
\]
  We now calculate
\[
  l'_0(x)-(1+\frac{1}{x}l_0(x)-l^2_0(x))=-\frac{3}{x(4x-3)^2}<0 \quad x\in(\frac{3}{4},\infty).
\]
  By Lemma \ref{lem:odeineq}, we have
\[
  l_0(x)\leq u_0(x), \quad x\geq x^\ast; \qquad
  l_0(x)\geq r_0(x), \quad x\in (\frac{3}{4},x^\ast).
\]
  Equivalently,
\[
  g_0(x)\geq 0, \quad x\geq x^\ast; \qquad g_0(x)<0, \quad x\in [\frac{3}{4},x^\ast),
\]
  completing the proof.
\end{proof}
\begin{remark}
  Numerical computation shows that $x^\ast\approxeq 1.452165365078841\ldots$.
\end{remark}
\begin{corollary}\label{h0pos}
  Let
  \be\label{h0def}
  h_0(x)=(x-2)g_0(x)=(x-2)[(4x-3)I_0(x)-(4x-1)I_1(x)],
  \ee
  where $g_0(x)$ is defined in \eqref{g0def}.
  Then $h_0(x)\geq 0$ on $[\frac{3}{4},x^\ast]$ and $[2,\infty)$;
    $h_0(x)\leq 0$ on $[x^\ast,2)$.
\end{corollary}
\begin{lemma}\label{h1pos}
  Let
  \be\label{h1def}
  h_1(x)=(4x^2-7x)I_1(x)-(4x^2-9x+3)I_0(x).
  \ee
  Then $h_1(x)>0$ on $[\frac{3}{4},\infty)$.
\end{lemma}
\begin{proof}
  Let $x_1^\ast=\frac{\sqrt{33}+9}{8}=1.843\ldots$ be the larger root of
  $4x^2-9x+3$. Then $4x^2-9x+3>0$ for $x>x_1^\ast$ and $4x^2-9x+3<0$ for
  $x\in[\frac{3}{4},x_1^\ast)$. We break $[\frac{3}{4},\infty)$ into several
  subintervals and show the positivity of $h_1(x)$ on each subinterval.
  \begin{enumerate}[label=(\alph*)]
  \item $x\in [x_1^\ast,\infty)$.
  Let
\[
  u_0(x)=\frac{4x^2-7x}{4x^2-9x+3}.
\]
  Then
  \be\label{u0diff}
  u'_0(x)-(1+\frac{1}{x}u_0(x)-u^2_0(x))=\frac{3\,\left(x-3\right)}{{\left(4\,x^2-9\,x+3\right)}^2},
  \ee
  which is greater than zero if $x>3$ and less than zero if $x<3$.
  At $x=3$, $u_0(3)=\frac{5}{4}=1.25$ and
  $r_0(3)= 1.23459\ldots<1.25=u_0(3)$. Thus,
  Using Lemma \ref{lem:odeineq} in the increasing direction we have $r_0(x)< u_0(x)$ on $[3,\infty)$;
  and using Lemma \ref{lem:odeineq} in the decreasing direction, we still have
  $r_0(x)< u_0(x)$ on $[x_1^\ast,3)$.
  Equivalently, $h_1(x)>0$ on $[x_1^\ast,\infty)$.
  
  \item $x\in[\frac{7}{4},x_1^\ast]$. On this subinterval, we have
  $4x^2-7x\geq 0$ and $-4x^2+9x-3\geq 0$. Hence, $h_1(x)>0$, since $I_1(x)$
  and $I_0(x)$ are always positive on $[0,\infty)$.

  \item $x\in[\frac{3}{4},\frac{7}{4}]$. By \eqref{u0diff}, we have
    $u'_0(x)-(1+\frac{1}{x}u_0(x)-u^2_0(x))\leq 0$ on $[\frac{3}{4},\frac{7}{4}]$.
    Also, $u_0(\frac{3}{4})=2<r_0(\frac{3}{4})=2.8\ldots$. Using Lemma
    \ref{lem:odeineq}, we have $r_0(x)>u_0(x)$, or equivalently $h_1(x)>0$ on
    $[\frac{3}{4},\frac{7}{4}]$.
  \end{enumerate}
\end{proof}
\begin{lemma}\label{rho0bound1}
  Let $f_0(x)=e^{-\frac{1}{x}}I_0(\frac{1}{x})$,
  $f_1(x)=e^{-\frac{1}{x}}I_1(\frac{1}{x})$, $f(x)=f_0(x)+a f_1(x)$ with
  $a=0.05$.
  Then $f'''(x)>0$ on $(0,\infty)$.
\end{lemma}
\begin{proof}
Using the recurrence relation~\eqref{inrec}, we obtain
\[
f_0'''(x)=\frac{1}{x^4}e^{-\frac{1}{x}}h_0\left(\frac{1}{x}\right),
\]
where $h_0(x)$ is defined in \eqref{h0def}.
Similarly,
\[
f_1'''(x)=\frac{1}{x^4}e^{-\frac{1}{x}}h_1\left(\frac{1}{x}\right),
\]
where $h_1(x)$ is defined in \eqref{h1def}.
Thus, in order to show that $f'''(x)>0$ on $(0,\infty)$, we only need
to show that $h_0(x)+a h_1(x)>0$ on $(0,\infty)$.

We break it into several steps.
\begin{enumerate}[label=(\alph*)]
\item $x\in[0,1/4]$.
  On this interval, $3-4x\geq 2$, $0\leq 1-4x\leq 1$, $2-x\geq 1.75$, thus $h_0(x)\geq 1.75(2I_0(x)-I_1(x))>1.75I_0(x)$.
  And $4x^2-7x\geq -1.5$, $4x^2-9x+3\leq 3$, thus $h_1(x)\geq-1.5I_1(x)-3I_0(x)>-4.5I_0(x)$. Combinging these results,
  we have
  \[
  h_0(x)+ah_1(x)>(1.75+0.05\times(-4.5))I_0(x)>0.
  \]
\item $\frac{1}{4}\leq x\leq\frac{9-\sqrt{33}}{8}<0.5$.
  On this interval, $3-4x>1$, $4x-1\geq 0$, $2-x>1.5$, thus $h_0(x)>1.5I_0(x)$.
  And $4x^2-7x>-2.5$, $0\leq 4x^2-9x+3\leq 1$, thus $h_1(x)>-2.5I_1(x)-I_0(x)>-3.5I_0(x)$. Combining these results,
  we have
  \[
  h_0(x)+ah_1(x)>(1.5+0.05\times (-3.5))I_0(x)>0.
  \]
\item $\frac{9-\sqrt{33}}{8} \leq x \leq 3/4$.
  On this interval, $3-4x\geq 0$, $4x-1>0.6$, $2-x>1$, thus $h_0(x)>0.6I_1(x)$.
  And $4x^2-7x\geq -3$, $-(4x^2-9x+3)\geq 0$, thus $h_1(x)\geq -3I_1(x)$. Combining these results,
  we have
  \[
  h_0(x)+ah_1(x)>(0.6-0.05\times 3)I_1(x)>0.
  \]
\item $x\in [\frac{3}{4},x^\ast] \cup [2,\infty)$. On these two subintervals,
  both $h_0(x)$ and $h_1(x)$ are positive by Corollary \ref{h0pos} and Lemma
  \ref{h1pos}. Thus $h_0(x)+ah_1(x)>0$.
\item $x\in (x^\ast,2)$.
  We calculate
  \[
  h'_1(x)=(x-3)g_0(x),
  \]
  where $g_0(x)$ is defined in \eqref{g0def}.
  By Lemma~\ref{g0sign}, $g_0(x)>0$ on $(x^\ast,\infty)$. Thus, $h'_1(x)< 0$
  on $(x^\ast,2)$. This shows that $h_1(x)>h_1(2) \approx 0.901688$ on $(x^\ast,2)$.
  On the other hand, it is straightforward to show that $g'_0(x)>0$ and $g''_0(x)<0$ on
  $(x^\ast,2)$. Hence, $h''_0(x)=g''_0(x)(x-2)+2g'_0(x)>0$ on $(x^\ast,2)$, indicating that
  $h_0(x)$ achieves its minimum at exactly one point.
  Numerical calculation shows that
  \[
  \min_{x\in (x^\ast,2)} h_0(x)\approx -0.043\ldots>-0.044.
  \]
  Hence,
  \[
  h_0(x)+a h_1(x)\geq \min_{x\in (x^\ast,2)} h_0(x)+0.05\times h_1(2)>0.
  \]
\end{enumerate}
\end{proof}

\bibliographystyle{abbrv}
\bibliography{heat}

\end{document}